\documentclass{amsart}
\usepackage{amsmath}
\usepackage{amssymb}
\usepackage{amsfonts}
\usepackage{mathrsfs}
\usepackage{latexsym}
\usepackage{amsbsy}
\usepackage[all]{xy}
\usepackage{xypic}
\usepackage{amscd}
\usepackage[dvips]{graphicx}
\usepackage{color}

\newtheorem{Thm}{Theorem}[section]
\newtheorem{Prop}[Thm]{Proposition}
\newtheorem{Lem}[Thm]{Lemma}
\newtheorem{Cor}[Thm]{Corollary}

\theoremstyle{definition}

\theoremstyle{remark}
\newtheorem{remark}[Thm]{Remark}
\newtheorem{Problem}[Thm]{Problem}

\DeclareMathOperator{\Ext}{Ext}
\DeclareMathOperator{\Hom}{Hom}

\DeclareMathOperator{\soc}{soc}
\DeclareMathOperator{\Top}{top}
\DeclareMathOperator{\dimv}{\underline{dim}}
\DeclareMathOperator{\rep}{rep}

\DeclareMathOperator{\modl}{mod}
\DeclareMathOperator{\ad}{ad}

\DeclareMathOperator{\Sym}{Sym}

\def\mfk{\mathfrak}
\def\mfkg{\mathfrak{g}}
\def\mb{\mathbb}
\def\mc{\mathcal}

\def\mchp{\mc{H}^+(\mathscr{A})}
\def\mchm{\mc{H}^-(\mathscr{A})}
\def\mca{\mathscr{A}}

\def\mcp{\mc{P}}
\def\mci{\mc{I}}

\def\mbc{\mb{C}}

\def\mbz{\mb{Z}}
\def\lra{\longrightarrow}
\def\ma{\mathscr{A}}
\def\ges{\geqslant}

\def\mo{\mathscr{O}}
\def\Coh{\text{Coh}}
\def\Cohx{\Coh(\mathbb{X})}

\def\dh{\mathbf{DH}}

\begin{document}

\title[Hall algebra approach to Drinfeld's presentation]{Hall algebra approach to Drinfeld's presentation of quantum loop algebras}

\author{Rujing Dou}
\address{Department of Mathematics, Tsinghua University, Beijing 100084, P.R.China}
\email{drj05@mails.thu.edu.cn}

\author{Yong Jiang}
\address{Department of Mathematics, Tsinghua University, Beijing 100084, P.R.China}
\curraddr{Fakult\"{a}t f\"{u}r Mathematik, Universit\"{a}t Bielefeld, Postfach 100131, 33609 Bielefeld, Germany}
\email{yjiang@math.uni-bielefeld.de}

\author{Jie Xiao}
\address{Department of Mathematics, Tsinghua University, Beijing 100084, P.R.China}
\email{jxiao@math.tsinghua.edu.cn}

\thanks{The research was supported in part by NSF of China (No. 10631010) and by
NKBRPC (No. 2006CB805905). Y. Jiang was also partly supported by SFB/TR 45 in Universit\"{a}t Bonn and SFB 701 in Universit\"{a}t Bielefeld, Germany}



\bigskip

\begin{abstract}
The quantum loop algebra $U_{v}(\mathcal{L}\mathfrak{g})$ was
defined as a generalization of the Drinfeld's new realization of the
quantum affine algebra to the loop algebra of any Kac-Moody algebra
$\mathfrak{g}$. It has been shown by Schiffmann that the Hall algebra of the category of coherent sheaves on
a weighted projective line is closely related to the quantum loop algebra
$U_{v}(\mathcal{L}\mathfrak{g})$, for some $\mathfrak{g}$ with a star-shaped Dynkin diagram.
In this paper we study Drinfeld's presentation of $U_{v}(\mathcal{L}\mathfrak{g})$
in the double Hall algebra setting, based on Schiffmann's
work. We explicitly find out a collection of generators of the double composition algebra $\mathbf{DC}(\Coh(\mathbb{X}))$
and verify that they satisfy all the Drinfeld relations.
\end{abstract}

\maketitle

\section{Introduction}

\subsection{}
Let $\mathfrak{g}$ be a Kac-Moody algebra, $U(\mathfrak{g})$ be its
universal enveloping algebra. The Drinfeld-Jimbo quantum group
$U_{v}(\mathfrak{g})$ is defined by a collection of generators and
relations (see \ref{subsec quantum group}), which is a certain
deformation of the Chevalley generators and Serre relations for
$U(\mathfrak{g})$. When $\mathfrak{g}$ is affine, it is
well-known that $\mathfrak{g}$ can be constructed as (a central
extension of) the loop algebra $\mathcal{L}\mathfrak{g}_{0}$ of some
simple Lie algebra $\mathfrak{g}_{0}$. In this case Drinfeld gave
another set of generators and relations of $U_{v}(\mathfrak{g})$
known as \textit{Drinfeld's new realization} of quantum affine
algebras. This new presentation can be treated as a certain
deformation of the loop algebra presentation of $\mathfrak{g}$.
The isomorphism of the two presentations of $U_{v}(\mathfrak{g})$ was proved by
Beck \cite{Be} (also see \cite{J}). One can define the quantum loop algebra $U_{v}(\mathcal{L}\mathfrak{g})$
for any Kac-Moody algebra $\mathfrak{g}$ as a generalization of Drinfeld's presentation for quantum affine algebras (see \ref{subsec quantum loop alg}).

\subsection{}\label{subsec problem}
The Ringel-Hall algebra approach to quantum groups developed since
1990's, which shows a deep relationship between Lie theory and
finite dimensional hereditary algebras. More precisely, let $Q$
be the quiver whose underlying graph is the Dynkin diagram of the
Kac-Moody algebra $\mathfrak{g}$. Consider the category of finite
dimensional representations of $Q$ over a finite field
$k=\mathbb{F}_{q}$, denoted by $\modl(kQ)$. Due to Ringel and Green
(\cite{R}, \cite{Gr}), the composition subalgebra of the Hall
algebra $\mathbf{H}(\modl(kQ))$ is isomorphic to the positive part
of the quantum group $U_{v}^{+}(\mathfrak{g})$ where $v$ specializes
to $\sqrt{q}$. This result was generalized to the whole quantum
group by using the technique of Drinfeld double for Hopf algebras
\cite{X} (see \ref{Ringel-Green-Xiao}).

Thus it is quite natural to consider the following problem:

\begin{Problem}\label{prob}
How to understand Drinfeld's presentation of quantum affine
algebras (and more generally, quantum loop algebras) in the Hall algebra setting?
\end{Problem}

One possible way to solve the problem for quantum affine algebras is to explain Beck's isomorphism in the language of Hall algebras.
For type $\tilde{A}$ Hubery has given the answer for the positive part
$U_{v}^{+}(\widehat{\mathfrak{sl}}_{n})$ using nilpotent
representations of cyclic quivers \cite{H}. But it seems not easy to
generalize his method to other types. We should also mention
that McGerty \cite{Mc} has given the Drinfeld generators for the
positive part $U_{v}^{+}(\widehat{\mathfrak{sl}}_{2})$ using
representations of the Kronecker quiver \cite{Mc}.

\subsection{}
On the other hand, in his remarkable paper \cite{K} Kapranov
observed that there are connections between the Hall algebra
of the category of coherent sheaves on a smooth projective curve
$X$ and Drinfeld's new realization of the quantum affine algebra. And
when $X$ is the projective line, he constructed an isomorphism between
a subalgebra of the Hall algebra and another positive part (compared with the
standard one, see \ref{subsec quantum group}) of
$U_{q}(\widehat{\mathfrak{sl}}_{2})$ (also see \cite{BK}).

This result was generalized by Schiffmann \cite{S} using Hall
algebras of the categories of coherent sheaves on weighted projective
lines (introduced in \cite{GL}) as following:

When the Dynkin diagram $\Gamma$ of $\mathfrak{g}$ is a star-shaped
graph, he defined a certain positive part $U_{v}(\hat{\mathfrak{n}})$ of the quantum loop algebra
$U_{v}(\mathcal{L}\mathfrak{g})$ (Note that there is no standard positive part of $U_{v}(\mathcal{L}\mathfrak{g})$
since in general the loop algebra $\mathcal{L}\mathfrak{g}$ is not a Kac-Moody
algebra). And he established an epimorphism from $U_{v}(\hat{\mathfrak{n}})$ to a subalgebra
$\mathbf{C}(\Coh\mathbb{X})$ of the Hall algebra
$\mathbf{H}(\Coh\mathbb{X})$, where $\mathbb{X}$ is a weighted
projective line associated to $\Gamma$. Moreover, when
$\mathbb{X}$ is of parabolic or elliptic type (the corresponding
$\mathfrak{g}$ is of finite or affine type), the epimorphism is an isomorphism (see Theorem \ref{thm Schiffmann}). This means
that the Hall algebras for weighted projective lines is the right framework to consider Problem \ref{prob} for general quantum loop algebras.

However, the problem was not completely solved in Schiffmann's work. Namely not all Drinfeld's generators and relations were explicitly found out in the corresponding Hall algebra. In fact the composition algebra $\mathbf{C}(\Coh(\mathbb{X}))$ is not generated by part of Drinfeld generators (some Chevalley generators are involved, see \ref{subsec Schiffmann's result} for details). Moreover, only the positive half $U_{v}(\hat{\mathfrak{n}})$ was linked to the Hall algebra. Thus Drinfeld's presentation for the whole quantum loop algebra has not been fully  understood yet.

\subsection{}
In this paper we study the problem in the double Hall algebra $\dh(\Coh\mathbb{X})$, which is the reduced Drinfeld double of the Hall algebra
$\mathbf{H}(\Coh\mathbb{X})$. We define a subalgebra $\mathbf{DC}(\Coh\mathbb{X})$ of $\dh(\Coh\mathbb{X})$, called double composition algebra, and show that a collection of generators of $\mathbf{DC}(\Coh\mathbb{X})$ satisfy all Drinfeld's relations (Theorem \ref{thm main}). Thus we have an epimorphism $\Xi$ from the whole quantum loop algebra $U_{v}(\mathcal{L}\mathfrak{g})$ to the double composition algebra $\mathbf{DC}(\Coh\mathbb{X})$.

Let us briefly explain our method. First we consider the generators
and relations in $\mathbf{H}(\Coh(\mathbb{X}))$. Note that we assume that the
Dynkin diagram of the Kac-Moody algebra $\mathfrak{g}$ is a star-shaped graph $\Gamma$ (see \ref{subsec star-shaped
graph}). Each branch of $\Gamma$ corresponds to a subalgebra isomorphic to
$U_{v}(\widehat{\mathfrak{sl}}_{n})$. Thus we can use the
results in \cite{H} to find out Drinfeld generators
in such subalgebras. For the central vertex we keep the
generators given in \cite{S}. Then it remains to check the relations (Some of them have been verified in \cite{S}, see \ref{subsec known
relations}).

To check all Drinfeld's relations in the double Hall algebra $\dh(\Coh\mathbb{X})$, the
key part is to verify the relations involving both positive and negative
part. For these relations we have to investigate the
comultiplication in details. This is much more complicated than the non-weighted
case $\Coh(\mathbb{P}^1)$, which has been studied in \cite{BK} \cite{BS}. However, we show that most terms appearing in the
comultiplication of a Drinfeld generator corresponding to the central vertex do not effect our calculation (see the proof of
lemma \ref{lem comult considered 1} and \ref{lem comult considered
2}). Thus the problem can be solved.

Note that in a recent work \cite{BS} of Burban and Schiffmann, the
case of $U_{v}(\widehat{\mathfrak{sl}}_{2})$ has already been studied in the double Hall algebra setting.
Our results coincide with theirs in this case. In particular, the epimorphism $\Xi$ is an isomorphism.

\subsection{}
The paper is organized as follows: In section \ref{sec Hall alg and
Drinfeld double} we recall basic notions of Hall algebra and
double Hall algebra, for details one can see \cite{R1} and \cite{X}.
In section \ref{sec Drinfeld presentation of quantum loop} we recall
the definition of quantum loop algebra, see \cite{S} or the Appendix of \cite{S0}.
We give a brief review of the theory of coherent sheaves on weighted projective lines in section
\ref{sec coh sheaf on w p l}, the main reference for this section is
\cite{GL}. The main result of the paper (Theorem \ref{thm main}) is
stated in section \ref{sec main result}. The proof of the main
theorem consists of the next three sections. More precisely, in
section \ref{sec relations in a tube} we prove the relations
satisfied by elements corresponding to torsion sheaves lying in a non-homogeneous tube, following
\cite{H}. Section \ref{sec relation of positive part} is devoted to
the proof of relations in the positive Hall algebra. The
remaining relations, especially those involving elements
from both positive and negative part, are proved in section \ref{sec
relation of the whole part}. In section \ref{sec remarks}
we give some remarks for the quantum affine algebras. In particular,
we explain that the homomorphism $\Xi$ given by our main theorem
is not induced from a derived equivalence functor for many cases.

\section{Hall algebras and their Drinfeld doubles}\label{sec Hall alg and Drinfeld double}

\subsection{Hereditary category}\label{subsec hereditary cat}
Let $k=\mb{F}_q$ be a finite field with $q$ elements, and
$\mathscr{A}$ be an essentially small abelian category. We assume that $\ma$ is
$k$-linear, Hom-finite and Ext-finite. That is, for all objects $X$,
$Y$ and $Z$ in $\ma$, the sets $\text{Hom}(X,Y)$ and
$\text{Ext}^1(X,Y)$ are finite dimensional $k$-vector spaces and the
composition of morphisms
$\text{Hom}(X,Y)\times\text{Hom}(Y,Z)\lra\text{Hom}(X,Z)$ is
$k$-bilinear. We also assume that $\ma$ is \emph{hereditary}, i.e.
$\text{Ext}^i(-,-)$ vanishes for all $i\ges 2$.

Let $\mcp$ be the set of isomorphism classes of objects in
$\mathscr{A}$ and $K_0(\mathscr{A})$ be the Grothedieck group of
$\mathscr{A}$. For any $\alpha\in\mcp$ we choose a representative
$V_{\alpha}\in\alpha$. And for each object $M$ in $\ma$, denote by
$[M]$ its image in $K_0(\mathscr{A})$. Then the \emph{Euler
form}
$$\langle[V_{\alpha}],[V_{\beta}]\rangle:=\dim_{k}\Hom_{\mathscr{A}}(V_{\alpha},V_{\beta})-
\dim_{k}\Ext_{\mathscr{A}}^{1}(V_{\alpha},V_{\beta})$$ is a
well-defined bilinear form on $K_0(\mathscr{A})$. The {\em symmetric Euler form} is defined by
$(\mu,\nu)$=$\langle\mu,\nu\rangle+\langle\nu,\mu\rangle$, for $\mu,\nu\in K_0(\mathscr{A})$.

\subsection{The Hall algebra}\label{subsec Hall alg}
For any $\alpha\in\mcp$, denote by $a_{\alpha}$ the cardinality of
the automorphism group of $V_{\alpha}$. Set $v=\sqrt{q}$.

For any $\alpha$, $\beta$ and $\gamma$ in $\mcp$, the {\em Hall
number} $g^{\gamma}_{\alpha\beta}$ is defined to be the number of
subobjects $X$ of $V_{\gamma}$ satisfying $X\in\beta$ and
$V_{\gamma}/X\in\alpha$.

The \emph{Hall algebra} associated to the category $\mca$, denoted
by $\mathbf{H}(\mca)$, is defined as follows: As a $\mbc$-vector space, it has a basis $\{u_{\alpha}|\alpha\in\mcp\}$. The multiplication is given by the following formula:
$$u_{\alpha}u_{\beta}=\sum_{\gamma\in\mcp}v^{\langle\alpha,\beta\rangle}g^{\gamma}_{\alpha\beta}u_{\gamma},\quad \forall\ \alpha,\beta\in\mcp.$$

It is easy to see that $\mathbf{H}(\mca)$ is an associative algebra with $1=u_{0}$.

\subsection{The extended Hall algebra}\label{subsec extended Hall alg}
The \emph{extended Hall algebra} $\mathcal{H}(\mca)$ is defined by adding the Grothendieck group $K_{0}(\mca)$ as a torus to the
Hall algebra $\mathbf{H}(\mca)$. More precisely, as a $\mbc$-vector space,
$\mathcal{H}(\mca)$ has a basis $\{K_{\mu}u_{\alpha}|\mu\in K_{0}(\mca),\ \alpha\in\mcp\}$. And the multiplication is given by the one inside $\mathbf{H}(\mca)$ together with the following additional rules:
$$K_{\mu}K_{\nu}=K_{\mu+\nu},\quad \forall\ \mu,\nu\in K_{0}(\mca),$$
$$K_{\mu}u_{\alpha}=v^{(\mu,\alpha)}u_{\alpha}K_{\mu},\quad \forall\ \mu\in K_{0}(\mca), \alpha\in\mcp.$$

If $\mca$ is a \emph{length category} (i.e. each object $M$ has a
filtration $0=M_{0}\subset M_{1}\subset\cdots\subset M_{n}=M$ such
that $M_{i+1}/M_{i}$ is a simple object for any $i$), then $\mathcal{H}(\mca)$ has a Hopf algebra structure with the following
comultiplication $\Delta$, counit $\epsilon$ and antipode $S$:
\begin{align*}
&\Delta(K_{\mu})=K_{\mu}\otimes K_{\mu},\\
&\Delta(u_{\gamma})=\sum_{\alpha,\beta\in\mcp}v^{\langle\alpha,\beta\rangle}
\frac{a_{\alpha}a_{\beta}}{a_{\gamma}}
g^{\gamma}_{\alpha\beta}u_{\alpha}K_{\beta}\otimes u_{\beta},\\
&\epsilon(u_{\alpha})=\delta_{\alpha,0},\ \epsilon(K_{\mu})=1,\\
&S(K_{\mu})=K_{-\mu}, S(u_{0})=0,\\
S(u_{\gamma})=\sum_{m\ge 1} (-1)^m &\sum_{\beta\in\mcp,\alpha_{1}\cdots\alpha_{m}\in\mcp_1} v^{2\sum_{i<j}\langle \alpha_{i},\alpha_{j}\rangle}
\frac{a_{\alpha_{1}}\cdots a_{\alpha_{m}}}{a_{\gamma}}g^{\gamma}_{\alpha_{1}\cdots\alpha_{m}}g^{\beta}_{\alpha_{1}\cdots\alpha_{m}}K_{-\gamma}u_\beta.
\end{align*}
where $\mcp_1=\mcp\setminus\{0\}$ and
$g^{\gamma}_{\alpha_{1}\cdots\alpha_{m}}$ is the number of all
filtrations
$$0=M_{0}\subset M_{1}\subset\cdots\subset M_{m}=V_{\gamma}$$ such that
$M_{i+1}/M_{i}\simeq V_{\alpha_{i}}$ for any $i$.

\begin{remark}
(1). The Hall algebra can be defined for any finitary abelian
category and it is both an algebra and a coalgebra. However, the heredity is needed to endow
the Hall algebra with the structure of a bialgebra. For details one can see Schiffmann's lecture notes
\cite{S0}.

(2). The comultiplication (resp. antipode) was first defined by
Green~\cite{Gr} (resp. Xiao~\cite{X}) in the case that $\mca$ is the
category of finite dimensional modules over a finite dimensional
hereditary algebra.

(3). If $\mathscr{A}$ is not a length category, then the comultiplication $\Delta$ takes value in the completion $\mathcal{H}(\mca)\widehat{\otimes}\mathcal{H}(\mca)$. And $\mathcal{H}(\mca)$ is a \emph{topological bialgebra} (see \cite{S1}).
\end{remark}

\subsection{The Drinfeld double of the Hall algebra}\label{subsec double Hall alg}
Now we write $\mchp$ for the extended Hall algebra
$\mathcal{H}(\mca)$ defined above. And we also write the basis
elements $u_{\alpha}^{+}$ instead of $u_{\alpha}$.

The ``negative" extended Hall algebra, denoted by $\mchm$, is defined
to be the $\mbc$-algebra with basis $\{K_{\mu}u_{\alpha}^{-}:\mu\in
K_{0}(\mca),\alpha\in\mcp\}$ and the following multiplication rules:
$$u_{\alpha}^{-}u_{\beta}^{-}=v^{\langle\alpha,\beta\rangle}\sum_{\gamma\in\mcp}
g^{\gamma}_{\alpha\beta}u_{\gamma}^{-},\
K_{\mu}K_{\nu}=K_{\mu+\nu},$$
$$K_{\mu}u_{\alpha}^{-}=v^{-(\mu,\alpha)}u_{\alpha}^{-}K_{\mu}.$$

Similarly, if $\mathscr{A}$ is a length category, $\mchm$ has a Hopf
algebra structure:
\begin{align*}
&\Delta(K_{\mu})=K_{\mu}\otimes K_{\mu},\\
&\Delta(u_{\gamma}^{-})=\sum_{\alpha,\beta\in\mcp}v^{\langle\beta,\alpha\rangle}
\frac{a_{\alpha}a_{\beta}}{a_{\gamma}}g^{\gamma}_{\beta\alpha}u_{\alpha}^{-}\otimes
u_{\beta}^{-}K_{-\alpha},\\
&\epsilon(u_{\alpha}^{-})=\delta_{\alpha,0},\ \epsilon(K_{\mu})=1,\\
&S(K_{\alpha})=K_{-\alpha}, S(u_{0}^{-})=0,\\
S(u_{\gamma}^{-})=\sum_{m\ge 1} (-1)^m &\sum_{\beta\in\mcp,\alpha_{1},\cdots,\alpha_{m}\in\mcp_1}v^{2\sum_{i<j}\langle\alpha_{i},\alpha_{j}\rangle}\frac{a_{\alpha_{1}}\cdots
a_{\alpha_{m}}}{a_{\gamma}} g^{\gamma}_{\alpha_{1}\cdots\alpha_{m}}
g^{\beta}_{\alpha_{m}\cdots\alpha_{1}}u_{\beta}^{-}K_{\gamma}.
\end{align*}

Actually in this case $\mchm$ is the dual Hopf algebra of $\mchp$
with opposite comultiplication.

Following Ringel~\cite{R1}, we define a bilinear form $\varphi
:\mchp\times\mchm\lra\mbc$ by
$$\varphi(K_{\mu}u_{\alpha}^{+},K_{\nu}u_{\beta}^{-})=v^{-(\mu,\nu)-(\alpha,\nu)+(\mu,\beta)}\frac{1}{a_{\alpha}}\delta_{\alpha\beta},$$
for any $\mu$, $\nu$ in $K_{0}(\mca)$ and $\alpha$, $\beta$ in
$\mcp$.

The bilinear form defined above is a skew-Hopf pairing on
$\mchp\times\mchm$. It induces a Hopf algebra structure on
$\mchp\otimes\mchm$ as follows: Let $\widetilde{\dh}(\mca)$ be the
free product of algebras $\mchp$ and $\mchm$ subject to the
following relations $D(u^-_\alpha,u^+_\beta)$ for all $u^-_\alpha\in
\mchm$ and $u^+_\beta\in \mchp$:
$$\sum_{i,j}a_i^{(1)-}b_j^{(2)+}\varphi(b_j^{(1)+},a_i^{(2)-})=
\sum_{i,j}b_j^{(1)+}a_i^{(2)-}\varphi(b_j^{(2)+},a_j^{(1)-}),$$
where $\Delta(u_{\alpha}^{-})=\sum_i a_i^{(1)-}\otimes a_i^{(2)-}$,
$\Delta(u_{\beta}^{+})=\sum_j b_j^{(1)+}\otimes b_j^{(2)+}$. Then $\widetilde{\dh}(\mca)$ is a (topological) Hopf algebra.

The \emph{double Hall algebra} of the category $\ma$ is defined to be
the quotient of $\widetilde{\dh}(\mca)$ by the ideal generated by $\{K_{\mu}\otimes 1-1\otimes
K_{-\mu}:\mu\in K_{0}(\mca)\}$. This algebra is called the
\emph{reduced Drinfeld double} of the pairing
$(\mchp,\mchm,\varphi)$, denoted by $\dh(\mca)$. For details see \cite{X}.

The double Hall algebra has the following triangular decomposition
$$\dh(\mca)=\mathbf{H}^{+}(\ma)\otimes\mathbf{T}\otimes\mathbf{H}^{-}(\ma),$$
where $\mathbf{T}$ (resp. $\mathbf{H}^{+}(\ma)$, $\mathbf{H}^{-}(\ma)$) is the subalgebra of $\dh(\mca)$ generated by
$\{K_{\mu}:\mu\in K_{0}(\mca)\}$ (resp. $\{u_{\alpha}^{+}:\alpha\in\mcp\}$,
$\{u_{\alpha}^{-}:\alpha\in\mcp\}$). It is easy to see that
$$\mathcal{H}^{-}(\mca)\simeq\mathbf{T}\otimes\mathbf{H}^{-}(\mca),
\mathcal{H}^{+}(\mca)\simeq\mathbf{H}^{+}(\mca)\otimes\mathbf{T}.$$

\subsection{The Hall algebra of $\modl kQ$}\label{Ringel-Green-Xiao}
Let $Q$ be a finite quiver, $I$ be the set of vertices of $Q$. Denote by
$\modl kQ$ the category of finite dimensional nilpotent
representations of $Q$ over $k$. This is a
hereditary length category. So we have the Hall algebra
$\mathbf{H}(\modl kQ)$ and the extended Hall algebra
$\mathcal{H}(\modl kQ)$. The \emph{composition subalgebra}
$\mathbf{C}(\modl kQ)$ is defined to be the subalgebra of
$\mathbf{H}(\modl kQ)$ generated by $u_{S_{i}}$ ($i\in I$), where
$S_{i}$ is the simple module in $\modl kQ$ corresponding to the
vertex $i$. The composition subalgebra provides a realization of the
positive part of the quantum group $U_{v}(\mathfrak{g})$ (see the next section for definitions):

\begin{Thm}[Ringel \cite{R}, Green \cite{Gr}]\label{thm Ringel-Green}
Let $\mathfrak{g}$ be the Kac-Moody algebra whose Dynkin diagram is the underlying graph of $Q$. Then we have an isomorphism
$$\mathbf{C}(\modl kQ)\simeq U_{v}^{+}(\mathfrak{g}),$$
where the image of the generator $u_{S_{i}}$ is precisely the Chevalley generator $E_{i}$ for each $i\in I$.
\end{Thm}

Furthermore, $\mathbf{C}(\modl kQ)\otimes\mathbf{T}$ is a Hopf
subalgebra of $\mathcal{H}(\modl kQ)$. So we can construct the
reduced Drinfeld double $\mathbf{DC}(\modl kQ)$, which is a
subalgebra of $\mathbf{DH}(\modl kQ)$. In this way the above Ringel-Green theorem is extended to the whole quantum group:

\begin{Thm}[\cite{X}]\label{thm Xiao}
The double composition algebra is isomorphic to the quantum group:
$$\mathbf{DC}(\modl kQ)\simeq U_{v}(\mathfrak{g}),$$
where $u_{S_{i}}^{+}\mapsto E_{i}$, $u_{S_{i}}^{-}\mapsto -v^{-1}F_{i}$, $K_{i}\mapsto K_{i}$.
\end{Thm}

\section{Drinfeld's presentation of quantum loop algebras}\label{sec Drinfeld presentation of quantum loop}
Throughout this section $v$ is an indeterminate. This should not cause confusion with the previous notation $v=\sqrt{q}$ as later we will consider quantum groups (resp. quantum loop algebras) with $v$ specialized to $\sqrt{q}$.

\subsection{Kac-Moody algebras}\label{subsec Kac-Moody alg}
Let $\mci$ be a finite set, $C=(c_{ij})_{i,j\in\mci}$ be a
generalized Cartan matrix. In this paper we only consider the case that $C$ is symmetric. let $\mfk{g}=\mfk{g}(C)$ be the
\emph{Kac-Moody algebra} associated to $C$, which is the complex Lie
algebra generated by $\{e_i,f_i,h_i:i\in\mci\}$ with relations
\begin{align*}
&[h_i,h_j]=0, \forall\ i,j\in\mci;\\
&[h_i, e_j]=c_{ij}e_j,\
[h_i,f_j]=-c_{ij}f_j, \forall \ i,j\in\mci;\\
&[e_i,f_j]=\delta_{ij}h_i, \forall\ i,j\in\mci;\\
&(\ad e_i)^{1-c_{ij}}e_{j}=0,\ (\ad f_i)^{1-c_{ij}}f_j=0,
\forall\ i\in\mci\ \text{and}\ j\in\mci\ \text{with}\ i\ne j.
\end{align*}

The root system of $\mfkg$ is denoted by $\Delta$. The simple roots
are denoted by $\alpha_{i}$, $i\in\mci$.
Let $Q=\oplus_{i\in\mci}\mathbb{Z}\alpha_{i}$ be the root lattice, which is equipped with the Cartan bilinear form
defined by $(\alpha_i, \alpha_j)=a_{ij}$.

\subsection{Quantum groups}\label{subsec quantum group}
First we recall the following standard notations:
$$[m]=\frac{v^{m}-v^{-m}}{v-v^{-1}},\quad \forall\ m\in\mathbb{Z},$$
$$[n]!=[n][n-1]\cdots[1],\quad \forall\ n\in\mathbb{N},$$
\begin{equation*}
\left[ \begin{array}{c} n\\t\\ \end{array}\right]=\frac{[n]!}{[t]![n-t]!},\quad \forall\ n,t\in\mathbb{N}, t\leq n.
\end{equation*}

The Drinfeld-Jimbo \emph{quantum group} (or the {\em quantized
enveloping algebra}) $U_v(\mfk{g})$ of a Kac-Moody algebra
$\mfk{g}$ is the $\mbc(v)$-algebra generated by
$\{E_i,F_i:i\in\mci\}$ and $\{K_{\mu}:\mu\in\mbz\mci\}$ with the
following defining relations:
\begin{align*}
&K_0=1,\ \ K_{\mu}K_{\nu}=K_{\mu+\nu},\ \ \forall\
\mu,\nu\in\mbz\mci;\\
&K_{\mu}E_i=v^{(\mu,i)}E_i K_{\mu},\ \ K_{\mu}F_i=v^{-(\mu,i)}F_i
K_{\mu},\ \forall\ i\in\mci,\ \mu\in\mbz\mci;\\
&E_iF_j-F_jE_i=\delta_{ij}\frac{K_i-K_{-i}}{v-v^{-1}},\ \forall\
i,j\in\mci;\\
&\sum_{p=0}^{1-c_{ij}}(-1)^p \left[ \begin{array}{c} 1-c_{ij} \\
p\\ \end{array}\right] E_i^p E_j E_i^{1-c_{ij}-p}=0,\ \forall\ \
i\neq j\in\mci;\\
&\sum_{p=0}^{1-c_{ij}}(-1)^p \left[ \begin{array}{c} 1-c_{ij} \\
p\\ \end{array}\right] F_i^p F_j F_i^{1-c_{ij}-p}=0,\ \forall\ \
i\neq j\in\mci.
\end{align*}

The quantized enveloping algebra has a natural triangular
decomposition:
$$U_v(\mfk{g})=U_v^-(\mfk{g})\otimes
U_v^0(\mfk{g})\otimes U_v^+(\mfk{g}),$$ where $U_{v}^{+}(\mfkg)$
(resp. $U_{v}^{-}(\mfkg)$, $U_{v}^{0}(\mfkg)$) is the subalgebra
generated by $E_i$ (resp. $F_i$, $K_{\mu}$).

\subsection{The Loop algebra of $\mfk{g}$}\label{subsec loop alg}
The \emph{loop algebra} of a Kac-Moody algebra $\mfkg$, denoted by $\mathcal{L}\mfkg$,
is defined to be the Lie algebra generated by $\{h_{i,k},
e_{i,k}, f_{i,k},c:i\in\mci, k\in\mathbb{Z}\}$ subject to the
following relations:
\begin{align*}
&c \text{ is central in } \mathcal{L}\mfkg,\\
&[h_{i,k},h_{j,l}]=k\delta_{k,-l}c_{ij}c,\\
&[e_{i,k},f_{j,l}]=\delta_{i,j}h_{i,k+l}+k\delta_{k,-l}c,\\
&[h_{i,k},e_{j,l}]=c_{ij}e_{j,l+k},\ [h_{i,k},f_{j,l}]=-c_{ij}f_{j,l+k},\\
&[e_{i,k+1},e_{j,l}]=[e_{i,k},e_{j,l+1}],\ [f_{i,k+1},f_{j,l}]=[f_{i,k},f_{j,l+1}],\\
&[e_{i,k_1},[e_{i,k_2},[\ldots,[e_{i,k_n},e_{j,l}]\ldots]=0,\
\text{for}\ n=1-c_{ij},\\
&[f_{i,k_1},[f_{i,k_2},[\ldots,[f_{i,k_n},f_{j,l}]\ldots]=0,\
\text{for}\ n=1-c_{ij}.
\end{align*}

It is clear that there is an embedding of Lie algebras
$\mfkg\hookrightarrow\mathcal{L}\mfkg$ assigning $e_{i}$, $f_{i}$,
$h_{i}$ to $e_{i,0}$, $f_{i,0}$, $h_{i,0}$ respectively.

Set $\widehat{Q}=Q\oplus\mathbb{Z}\delta$. We can extend the Cartan
form to $\widehat{Q}$ by setting $(\delta,\alpha)=0$ for all
$\alpha\in\widehat{Q}$. $\mathcal{L}\mfk{g}$ is $\widehat{Q}$-graded
by setting $\deg(e_{i,k})=\alpha_i+k\delta$,
$\deg(f_{i,k})=-\alpha_i+k\delta$ and $\deg(h_{i,k})=k\delta$. The root
system of $\mathcal{L}\mfk{g}$ is $\widehat{\Delta}=\mathbb{Z}^*\delta\cup
\{\Delta+\mathbb{Z}\delta\}$ (See \cite{MRY}).

For each root $\alpha\in\widehat{Q}$, we call it \emph{real} if
$(\alpha, \alpha)=2$, and \emph{imaginary} if $(\alpha, \alpha)\leq
 0$.

\begin{remark}
(1) If $\mfkg$ is a complex simple Lie algebra, $\mathcal{L}\mfkg$
is isomorphic to $\widehat{\mfkg}=\mfkg[t,t^{-1}]\oplus\mathbb{C}c$,
which is an affine Kac-Moody algebra. The assignment $e_{i,k}\mapsto e_{i}t^{k}$, $f_{i,k}\mapsto f_{i}t^{k}$,
$h_{i,k}\mapsto h_{i}t^{k}$ and $c\mapsto c$ gives the isomorphism (See \cite{Ga}).

(2) In general, $\mathcal{L}\mfkg$ and $\widehat{\mfkg}$ are not
Kac-Moody algeras. And the above assignment only yields a surjective
homomorphism $\mathcal{L}\mfkg\rightarrow\widehat{\mfkg}$, whose
kernel was described in \cite{MRY}.
\end{remark}

\subsection{Quantum loop algebras}\label{subsec quantum loop alg}
The \emph{quantum loop algebra} (with zero central charge)
$U_v(\mathcal{L}\mfk{g})$ is the $\mathbb{C}(v)$-algebra generated
by $x^\pm_{i,k}$, $h_{i,l}$ and $K^{\pm 1}_i$ for $i\in\mci,
k\in\mathbb{Z}$, and $l\in\mathbb{Z}^*$ subject to the following
relations:
\begin{gather}
[K_i,K_j]=[K_i,h_{j,l}]=0,\\
[h_{i,l},h_{j,k}]=0,\\
K_ix^\pm_{j,k}K^{-1}_i=v^{\pm a_{ij}}x^\pm_{j,k},\\
[h_{i,l},x^\pm_{j,k}]=\pm\frac{1}{l}[la_{ij}]x^\pm_{j,k+l},\\
x^\pm_{i,k+1}x^\pm_{j,l}-v^{\pm a_{ij}}x^\pm_{j,l}x^\pm_{i,k+1}=v^{\pm a_{ij}}x^\pm_{i,k}x^\pm_{j,l+1}-x^\pm_{j,l+1}x^\pm_{i,k},
\end{gather}
\begin{equation}
\Sym_{k_1,\ldots,k_n}\sum^n_{t=0}(-1)^t\left[ \begin{array}{c} n \\
t\\ \end{array}\right]x^\pm_{i,k_1}\cdots
x^\pm_{i,k_t}x^\pm_{j,l}x^\pm_{i,k_{t+1}}\cdots x^\pm_{i,k_n}=0,
\end{equation}
where $i\neq j$, $n=1-a_{ij}$ and $\Sym_{k_1,\ldots,k_n}$ denotes
symmetrization with respect to the indices $k_1,\ldots,k_n$. And
\begin{equation}
[x^+_{i,k},x^-_{j,l}]=\delta_{ij}\frac{\psi_{i,k+l}-\varphi_{i,k+l}}{v-v^{-1}},
\end{equation}
where $\psi_{i,k}$ and $\varphi_{i,k}$ are defined by the following equations:
\begin{gather*}
\sum_{k\geq
0}\psi_{i,k}u^k=K_i\exp((v-v^{-1})\sum^\infty_{k=1}h_{i,k}u^k),\\
\sum_{k\geq
0}\varphi_{i,-k}u^{-k}=K^{-1}_i\exp(-(v-v^{-1})\sum^\infty_{k=1}h_{i,-k}u^{-k}).
\end{gather*}

\begin{remark}
If $\mathfrak{g}$ is a simple Lie algebra, the above definition of $U_{v}(\mathcal{L}\mfk{g})$ is the so-called Drinfeld's new realization of quantum affine algebras. Namely in this case, $U_v(\mathcal{L}\mfk{g})$ is isomorphic to the
Drinfeld-Jimbo quantized enveloping algebra $U_{v}(\widehat{\mfkg})$ (See \cite{Dr}, \cite{Be} for details).
\end{remark}

\section{The category of coherent sheaves on weighted projective lines}\label{sec coh sheaf on w p l}
Now we introduce the category of coherent sheaves on weighted projective lines as studied
in~\cite{GL}. In this section $k$ denotes an arbitrary field.

\subsection{Weighted projective lines}\label{subsec weighted projective line}
Let $\mathbf{p}=(p_{1},p_{2},\ldots,p_{n})\in\mathbb{N}^{n}$.
Consider the $\mathbb{Z}$-module
$$L(\mathbf{p})=\mathbb{Z}\vec{x}_{1}\oplus\mathbb{Z}\vec{x}_{2}\oplus\cdots\oplus\mathbb{Z}\vec{x}_{n}/J$$
where $J$ is the submodule generated by
$\{p_{1}\vec{x}_{1}-p_{s}\vec{x}_{s}|s=2,\ldots,n\}$. Set
$\vec{c}=p_{1}\vec{x}_{1}=\cdots=p_{n}\vec{x}_{n}\in L(\mathbf{p})$.

The polynomial ring $k[X_{1},\ldots,X_{n}]$ has a structure of
$L(\mathbf{p})$-graded algebra by setting $\deg X_{i}=\vec{x}_{i}$.
We will denote it by $S(\mathbf{p})$.

Let $\underline{\lambda}=\{\lambda_{1},\ldots,\lambda_{n}\}$ be a
collection of distinct closed points (of degree $1$) on the
projective line $\mathbb{P}^{1}(k)$. Let
$I(\mathbf{p},\underline{\lambda})$ be the $L(\mathbf{p})$-graded
ideal of $S(\mathbf{p})$ generated by
$\{X^{p_{s}}-(X^{p_{2}}-\lambda_{s}X^{p_{1}})|s=3,\ldots,n\}$. Set
$S(\mathbf{p},\underline{\lambda})=S(\mathbf{p})/I(\mathbf{p},\underline{\lambda})$,
which is also an $L(\mathbf{p})$-graded algebra. For any $i$, we denote by $x_{i}$ the image of
$X_{i}$ in $S(\mathbf{p},\underline{\lambda})$.

Let $\mathbb{X}_{\mathbf{p},\underline{\lambda}}=\text{Specgr}S(\mathbf{p},\underline{\lambda})$
be the set of all prime homogeneous ideals in
$S(\mathbf{p},\underline{\lambda})$. This is the so-called
\emph{weighted projective line}. The pair $(\mathbf{p},\underline{\lambda})$
is called the \emph{weight type} of
$\mathbb{X}_{\mathbf{p},\underline{\lambda}}$. The number $p_{i}$ is
the \emph{weight} of the point $\lambda_{i}$.

In the following we will fix a weight type $(\mathbf{p},\underline{\lambda})$ and write $S=S(\mathbf{p},\underline{\lambda})$,
$\mathbb{X}=\mathbb{X}_{\mathbf{p},\underline{\lambda}}$ for short.

\subsection{Coherent sheaves on $\mathbb{X}_{\mathbf{p},\underline{\lambda}}$}\label{subsec coherent sheaf}
For any homogeneous element $f\in S$, let $V_{f}=\{\mathfrak{p}\in\mathbb{X}|f\in\mathfrak{p}\}$ and
$D_{f}=\mathbb{X}\setminus V_{f}$.

The \emph{structure sheaf} $\mo_{\mathbb{X}}$ is defined to be the
sheaf of $L(\mathbf{p})$-graded algebras on $\mathbb{X}$ associated to
the presheaf $D_{f}\mapsto S_{f}$, where $S_{f}=\{g/f^{l}|g\in
S,l\in\mathbb{N}\}$. We denote by $\mo_{\mathbb{X}}$-Mod the
category of sheaves of $L(\mathbf{p})$-graded
$\mo_{\mathbb{X}}$-modules on $\mathbb{X}$.

For any $\vec{x}\in L(\mathbf{p})$ and any $L(\mathbf{p})$-graded
$\mo_{\mathbb{X}}$-module $\mathscr{M}$, we denote by
$\mathscr{M}(\vec{x})$ the \emph{shift} of $\mathscr{M}$ by
$\vec{x}$ (i.e. $\mathscr{M}(\vec{x})[\vec{y
}]=\mathscr{M}[\vec{x}+\vec{y}]$). A sheaf $\mathscr{M}$ of
$L(\mathbf{p})$-graded $\mo_{\mathbb{X}}$-module is called
\emph{coherent} if there exists an open covering $\{U_{i}\}$ of
$\mathbb{X}$ and for each $i$ an exact sequence
$$\bigoplus_{s=1}^{N}\mo_{\mathbb{X}}(\vec{l_{s}})|_{U_{i}}\rightarrow\bigoplus_{t=1}^{M}\mo_{\mathbb{X}}(\vec{k_{t}})|_{U_{i}}\rightarrow\mathscr{M}|_{U_{i}}\rightarrow 0.$$

The category of coherent sheaves on $\mathbb{X}$, denoted by
$\Coh(\mathbb{X})$, is a full subcategory of $\mo_{\mathbb{X}}$-Mod. It has been proved in \cite{GL} that $\Coh(\mathbb{X})$ is a $k$-linear
hereditary, Hom- and Ext-finite abelian category.

\subsection{The structure of the category $\Coh(\mathbb{X})$}\label{subsec structure of Coh}
Let $\mathscr{F}$ be the full subcategory of $\Coh(\mathbb{X})$
consisting of all locally free sheaves, and $\mathscr{T}$ be the
full subcategory consisting of all torsion sheaves. Both $\mathscr{F}$ and
$\mathscr{T}$ are extension-closed. Moreover, $\mathscr{T}$ itself is a hereditary length abelian category. The following lemma was
proved in \cite{GL}:

\begin{Lem}\label{lem free and torsion}
(1). For any sheaf $\mathscr{M}\in\Coh(\mathbb{X})$, it can be
decomposed as $\mathscr{M}_{t}\oplus\mathscr{M}_{f}$ where
$\mathscr{M}_{t}\in\mathscr{T}$ and $\mathscr{M}_{f}\in\mathscr{F}$.

(2).
$\rm{Hom}(\mathscr{M}_{t},\mathscr{M}_{f})=\rm{Ext^{1}}(\mathscr{M}_{f},\mathscr{M}_{t})=0$,
for any $\mathscr{M}_{t}\in\mathscr{T}$ and
$\mathscr{M}_{f}\in\mathscr{F}$.
\end{Lem}

To describe $\mathscr{T}$ more precisely, we need a classification
of the closed points in $\mathbb{X}$. Recall that
$\mathbb{X}=\text{Specgr}S(\mathbf{p},\underline{\lambda})$.
According to \cite{GL}, each $\lambda_{i}$ corresponds to the prime
ideal generated by $x_{i}$, called an \emph{exceptional point}. And
other homogeneous primes are given by a prime homogeneous polynomial
$F(x_{1}^{p_{1}},x_{2}^{p_{2}})\in k[T_{1},T_{2}]$, which are called
\emph{ordinary points}.

Let $p=\text{l.c.m}(p_{1},\ldots,p_{n})$. The \emph{degree} of a
closed point is defined by setting $\deg(\lambda_{i})=p/p_{i}$ for any $i$
and $\deg(x)=pd$ for any ordinary point $x$ corresponding to a prime
homogeneous polynomial of degree $d$.

Let $C_{r}$ be the cyclic quiver with $r$ vertices. More precisely,
for $r=1$, $C_{1}$ is just the quiver with only one vertex and one
loop arrow. For $r\geq 2$, the vertices of $C_{r}$ are indexed by
$\mathbb{Z}/r\mathbb{Z}$ and the arrows are from $i$ to $i-1$ for
each $i$. Denote by $\rep_{0}(C_{r})_{k}$ the category of finite-dimensional nilpotent
representations of $C_{r}$ over the field $k$. The following lemma,
due to \cite{GL}, describes the structure of the subcategory
$\mathscr{T}$.

\begin{Lem}\label{lem structure of torsion}
(1). The category $\mathscr{T}$ decomposes as a coproduct
$\mathscr{T}=\coprod_{x\in\mathbb{X}}\mathscr{T}_{x}$, where
$\mathscr{T}_{x}$ is the subcategory of torsion sheaves with support
at $x$.

(2). For any ordinary point $x$ of degree $d$, let $k_{x}$
denote the residue field at $x$. Then $\mathscr{T}_{x}$ is
equivalent to the category $\rep_{0}(C_{1})_{k_{x}}$.

(3). For any exceptional point $\lambda_{i}$ ($1\leq i\leq
n$), the category $\mathscr{T}_{\lambda_{i}}$ is equivalent to $\rep_{0}(C_{p_{i}})_{k}$.
\end{Lem}

\subsection{Indecomposable objects in $\mathscr{T}$}\label{subsec indecomposable torsion sheaf}
We first give a description of the simple objects in $\mathscr{T}$.

For any ordinary point $x$ of degree $d$, let $\pi_{x}$ denote the prime
homogeneous polynomial corresponding to $x$. Then multiplication by
$\pi_{x}$ gives the exact sequence
$$0\rightarrow\mo_{\mathbb{X}}\rightarrow\mo_{\mathbb{X}}(d\vec{c})\rightarrow S_{x}\rightarrow 0.$$
$S_{x}$ is the unique (up to isomorphism) simple sheaf in the
category $\mathscr{T}_{x}$. Moreover, for any $\vec{k}\in
L(\mathbf{p})$ we have $S_{x}(\vec{k})=S_{x}$.

For any exceptional point $\lambda_{i}$, multiplication by $x_{i}$
yields the exact sequence
$$0\rightarrow\mo_{\mathbb{X}}((j-1)\vec{x_{i}})\rightarrow\mo_{\mathbb{X}}(j\vec{x_{i}})\rightarrow S^{i}_{j}\rightarrow 0,\ \text{for each }j,\ 1\leq j\leq p_{i}$$
And $\{S^{i}_{j}|1\leq j\leq p_{i}\}$ is a complete set of pairwise
non-isomorphic simple sheaves in the category
$\mathscr{T}_{\lambda_{i}}$, for any $i$ ($1\leq i\leq n$).
Moreover, for any $\vec{k}=\sum k_{i}\vec{x_{i}}$ we have
$S^{i}_{j}(\vec{k})=S^{i}_{j+k_{i}(\modl p_{i})}$.

Now we describe the indecomposable objects. Recall the following
well-known results on representation theory of cyclic quivers:

(1). In the category $\rep_{0}(C_{1})$, the set of isomorphism
classes of indecomposables is $\{S(a)|a\in\mathbb{N}\}$, where $S=S(1)$ is
the only simple representation, and $S(a)$ is the unique
indecomposable representation of length $a$.

For any partition
$\mu=(\mu_{1}\geq\cdots\geq\mu_{t})$, let
$S(\mu)=\bigoplus_{i}S(\mu_{i})$. Then any object in $\rep_{0}(C_{1})$ is isomorphic to $S(\mu)$ for some $\mu$.

(2). In the category $\rep_{0}(C_{r})$, we have $r$ simple
representations $S_{j}=S_{j}(1)$ ($1\leq j\leq r$). Denote by $S_{j}(a)$
the unique indecomposable representation with top $S_{j}$ and
length $a$. Then $\{S_{j}(a)|1\leq j\leq r,a\in\mathbb{N}\}$ is the set of all isomorphism
classes of indecomposables in $\rep_{0}(C_{r})$.

The above results, combined with lemma \ref{lem structure of
torsion}, give a classification of indecomposable objects in $\mathscr{T}$.
We denote by $S_{x}(a)$ the unique
indecomposable object of length $a$ in $\mathscr{T}_{x}$, for any
ordinary point $x$. And for any exceptional point $\lambda_{i}$, the
indecomposable objects in $\mathscr{T}_{\lambda_{i}}$ are denoted by
$S^{i}_{j}(a)$ ($1\leq j\leq p_{i},a\in\mathbb{N}$).

\subsection{The Grothendiek group and the Euler form}\label{subsec Euler form}
The following proposition (see \cite{GL}) gives an explicit
description of the Grothendieck group of $\Coh(\mathbb{X})$.

\begin{Prop}\label{prop Grothendieck group}
$$K_{0}(\Coh(\mathbb{X}))\cong(\mathbb{Z}[\mo_{\mathbb{X}}]\oplus\mathbb{Z}[\mo_{\mathbb{X}}(\vec{c})]\oplus
\bigoplus_{i,j}\mathbb{Z}[S^{i}_{j}])/I$$ where $I$ is the subgroup
generated by
$\{\sum_{j=1}^{p_{i}}[S^{i}_{j}]+[\mo_{\mathbb{X}}]-[\mo_{\mathbb{X}}(\vec{c})]|i=1,\ldots,n\}$.
\end{Prop}

In the following we simply write $\mo$ for $\mo_{\mathbb{X}}$. Set
$$\delta=[\mo(\vec{c})]-[\mo]=\sum_{j=1}^{p_{i}}[S^{i}_{j}], \text{ for } i=1,\ldots,n.$$
To calculate the Euler form on $K_{0}(\Coh(\mathbb{X}))$, we have
the following lemma (see \cite{S}):

\begin{Lem}\label{lem comute Euler form}
The Euler form $\langle-,-\rangle$ on $K_{0}(\Coh(\mathbb{X}))$ is given by
\begin{gather*}
\langle[\mo],[\mo]\rangle=1,\ \ \ \langle\mo,\delta\rangle=1,\ \ \
\langle\delta,[\mo]\rangle=-1,\\
\langle\delta,\delta\rangle=0,\ \ \
\langle\delta,[S^{i}_{j}]\rangle=0,\ \ \ \langle
[S^{i}_{j}],\delta\rangle=0,\\
\langle[\mo],[S^{i}_{j}]\rangle=\begin{cases}
                            1& \text{ if } j=p_{i}\\
                            0& \text{ if } j\neq p_{i}
                            \end{cases}\\
\langle[S^{i}_{j}],[\mo]\rangle=\begin{cases}
                             -1& \text{ if } j=1\\
                             0& \text{ if } j\neq 1
                             \end{cases}\\
\langle[S^{i}_{j}],[S^{i'}_{j'}]\rangle=\begin{cases}
                                     1&  \text{ if } i=i',j=j'\\
                                     -1& \text{ if } i=i',j\equiv
                                     j'+1 (\modl p_{i})\\
                                     0& \text{ otherwise}
                                     \end{cases}
\end{gather*}
\end{Lem}

\section{Main results}\label{sec main result}

\subsection{}\label{subsec Hall alg of Coh}
From now on we fix a finite field $k=\mathbb{F}_{q}$ and set
$v=\sqrt{q}$. And we also fix a weight type (see \ref{subsec
weighted projective line}):
$$\mathbf{p}=(p_{1},\cdots,p_{n}),\ \
\underline{\lambda}=\{\lambda_{1},\cdots,\lambda_{n}\}.$$ Consider
the weighted projective line
$\mathbb{X}=\mathbb{X}_{\mathbf{p},\underline{\lambda}}$ and the
category of coherent sheaves $\Coh(\mathbb{X})$. We keep the notions
in the last section.

Since $\Coh(\mathbb{X})$ is a $k$-linear hereditary, Hom- and
Ext-finite abelian category, we have the Hall algebra
$\mathbf{H}(\Cohx)$ and the double Hall algebra $\dh(\Cohx)$, as in
section \ref{subsec Hall alg} and \ref{subsec double Hall alg}.

The subcategory $\mathscr{T}$ (resp. $\mathscr{T}_{x}$, for any closed point
$x$ in $\mathbb{X}$) is a hereditary length abelian category. Thus we can
define the Hall algebra $\mathbf{H}(\mathscr{T})$ (resp.
$\mathbf{H}(\mathscr{T}_{x})$). It is clear that $\mathbf{H}(\mathscr{T})$ and
$\mathbf{H}(\mathscr{T}_{x})$ are sub-Hopf algebras of
$\mathbf{H}(\Cohx)$, since the categories $\mathscr{T}$ and $\mathscr{T}_{x}$
are stable under taking extensions, subobjects and quotients.

\subsection{}\label{subsec homogeneous tube}
In this and the next subsection we define some element $T_{r}\in\mathbf{H}(\mathscr{T})$, for any $r\in\mathbb{N}$,
following \cite{S}.

For any ordinary point $x$ of degree $d$, we know that
$\mathscr{T}_{x}$ is equivalent to the category of finite-dimensional nilpotent
representations of the quiver $C_{1}$ over $k_{x}$ (see Lemma
\ref{lem structure of torsion}). Since $k=\mathbb{F}_{q}$ we know that
$k_{x}=\mathbb{F}_{q^{d}}$. Thus we have the following isomorphism:
$$\Theta_{x}:\mathbf{H}(C_{1})_{\mathbb{F}_{q^{d}}}\rightarrow\mathbf{H}(\mathscr{T}_{x}),$$
where $\mathbf{H}(C_{1})_{\mathbb{F}_{q^{d}}}$ is the Hall algebra
associated to the the category $\rep_{0}(C_{1})_{\mathbb{F}_{q^{d}}}$.

let $\Lambda=\mathbb{C}[z_{1},z_{2},\cdots]^{\mathfrak{S}_{\infty}}$
be Macdonald's ring of symmetric functions (see \cite{Ma}). The
following result is well-known, due to P. Hall.

\begin{Thm}\label{thm symm funct iso hall alg}
There exists an isomorphism of algebras $\Upsilon: \Lambda\simeq
\mathbf{H}(C_1)_{\mathbb{F}_{q}}$ given by $e_{r}\mapsto v^{r(r-1)}S_{x}(1^r)$,
for any $r\geq 1$, where $e_{r}$ is the $r$-th elementary symmetric function.
\end{Thm}

Let $p_r$ denote the $r$-th power-sum symmetric function and set
$\mathbf{h}_r=\frac{[r]}{r}\Upsilon(p_r)$. Then
$$\mathbf{h}_r=\frac{[r]}{r}\sum_{|\mu|=r}n(l(\mu)-1)S_{x}(\mu),$$ where
$n(l)=\Pi^{l}_{i=1}(1-v^{2i})$ (see \cite{Ma}).

For each $r\in\mathbb{N}$, set
\begin{equation*}
\mathbf{h}_{r,x}=\begin{cases}
                 0& \text{if } r\nmid d\\
                 \Theta_{x}(\mathbf{h}_{r/d})& \text{if } r\mid d
                 \end{cases}
\end{equation*}

\subsection{}\label{subsec element T_r}
For any exceptional point $\lambda_{i}$, $\mathscr{T}_{\lambda_{i}}$
is equivalent to the category $\rep_{0}(C_{p_{i}})_{k}$. Thus we have
an isomorphism
$$\Theta_{\lambda_{i}}:\mathbf{H}(C_{p_{i}})_{k}\rightarrow\mathbf{H}(\mathscr{T}_{\lambda_{i}}),$$
where $\mathbf{H}(C_{p_{i}})_{k}$ is the Hall algebra associated to the category $\rep_{0}(C_{p_{i}})_{k}$.

For any positive integer $m$, there is a natural fully faithful functor $\iota_m:
\rep_{0}(C_m)\rightarrow\rep_{0}(C_{m+1})$ whose image is the full
subcategory consisting of all objects $X$ such that
$\Hom(X,S_m)=\Hom(S_{m+1},X)=0$. The functor $\iota_m$ induces an embedding of Hall algebras
$\mathbf{H}(C_m)\rightarrow\mathbf{H}(C_{m+1})$:
$$u_{S_i}\mapsto\begin{cases}
              u_{S_i} &\text{for}\  1\leq i<m\\
              u_{S_{m+1}(2)}=vu_{S_{m+1}}u_{S_m}-u_{S_m}u_{S_{m+1}}
              &\text{for}\ i=m
              \end{cases}
$$

Hence the composition $\iota_{m-1}\circ\cdots\circ\iota_{2}\circ\iota_{1}$ of functors gives an embedding of categories
$\rep_{0}(C_1)\hookrightarrow\rep_{0}(C_m)$, which induces an
embedding of Hall algebras:
$$\Psi:\mathbf{H}(C_1)\longrightarrow\mathbf{H}(C_{m}).$$

Set $$\mathbf{h}_{r,\lambda_{i}}=
\Theta_{\lambda_{i}}\circ\Psi(\mathbf{h}_r).$$

Finally we define
$$T_r=\sum_{x\in\mathbb{X}}\mathbf{h}_{r,x}\in\mathbf{H}(\Coh(\mathbb{X})),$$
where the sum is taken over all closed points $x$ on $\mathbb{X}$.

\subsection{}\label{subsec more about non-homogeneous tube}
We also need some notations in \cite{H} for the Hall algebra of $\mathbf{H}(C_{m})$.
For any $1\leq l\leq m$, let $\mathcal{M}_{l,\alpha}$ be
the set of all isomorphism classes of modules $M$ in
$\rep_{0}(C_{m})$ such that $\dimv M=\alpha$ and $\soc(M)\subseteq
S_1\oplus\cdots\oplus S_l$.

Let $\delta_{m}$ be the sum of dimension vectors of all simple
modules, i.e. the minimal imaginary root for $C_{m}$. For any $r\in\mathbb{N}$, set
$$c_{l,r}=(-1)^{r}v^{-2lr}\sum_{M\in\mathcal{M}_{l,r\delta_m}}(-1)^{\dim\rm{End(M)}}a_{M}u_{M}\in\mathbf{H}(C_{m}).$$

Then define $p_{l,r}\in\mathbf{H}(C_{m})$ via the following
generating function
$$\sum_{r\geq1}(1-v^{-2lr})p_{l,r}T^{r-1}=\frac{d}{dT}\log C_l(T),$$
where $C_l(T)=1+\sum_{r\geq1}c_{l,r}T^r$.

And define
$$\pi_{l,r}:=\frac{[lr]}{r}p_{l,r}.$$

For any exceptional point $\lambda_{i}$ on $\mathbb{X}$, let $\pi^{i}_{j,k}=\Theta_{\lambda_{i}}(\pi_{j,k})$.
It is not difficult to see that $\pi_{1,r}^{i}=\mathbf{h}_{r,\lambda_{i}}$.

\begin{remark}\label{rem center of HCm}
It was proved in \cite{H1} that there is an algebra isomorphism
$\mathbf{H}(C_m)\simeq U^{+}_v(\widehat{sl}_m)\bigotimes\mathcal{Z}$, where
$\mathcal{Z}=\mathbb{C}[p_{m,1},p_{m,2},...]$ is a central
subalgebra of $\mathbf{H}(C_m)$ and the element $p_{m,r}$ is
homogeneous of degree $r\delta_{m}$.
\end{remark}

\subsection{}\label{subsec star-shaped graph}
We can associate a star-shaped graph $\Gamma$ to the weight type $(\mathbf{p},\underline{\lambda})$:\\

\         
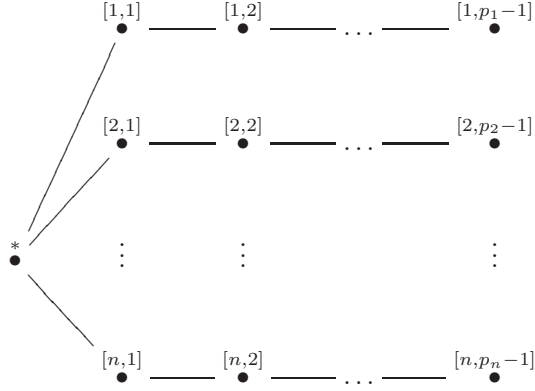
\xymatrix{&\stackrel{[1,1]}{\bullet}\rline & \stackrel{[1,2]}{\bullet}\rline & \ldots\rline &\stackrel{[1,p_1-1]}{\bullet}\\
&\stackrel{[2,1]}{\bullet}\rline & \stackrel{[2,2]}{\bullet}\rline & \ldots\rline &\stackrel{[2,p_2-1]}{\bullet}\\
\stackrel{\ast}{\bullet}\uurline\drline\urline&\vdots&\vdots& &\vdots\\
&\stackrel{[n,1]}{\bullet}\rline & \stackrel{[n,2]}{\bullet}\rline &
\ldots\rline &\stackrel{[n,p_n-1]}{\bullet} }\\

As marked in the graph, the central vertex is denoted by $\ast$.
There are $n$ branches and in each branch there are $p_{i}-1$
($1\leq i\leq n$) vertices respectively. We denote by $[i,j]$ the $j$th vertex in the
$i$th branch. Let $\Gamma_{0}=\{\ast,[i,j]|1\leq i\leq n, 1\leq j\leq p_{i}-1\}$ denote the set of vertices of $\Gamma$.

Consider the Kac-Moody algebra $\mfk{g}=\mfk{g}(\Gamma)$ associated
to the graph $\Gamma$. As in \ref{subsec loop alg} and \ref{subsec
quantum loop alg}, we have the loop algebra $\mathcal{L}\mfk{g}$ and
its quantized enveloping algebra $U_{v}(\mathcal{L}\mfk{g})$. The
root systems of $\mfk{g}$ and $\mathcal{L}\mfk{g}$ are denoted by
$\Delta$ and $\hat{\Delta}$ respectively. In view of the graph
$\Gamma$, the simple roots in $\Delta$ are denoted by
$\alpha_{\ast}$ and $\alpha_{ij}$, for $1\leq i\leq n$ and $1\leq
j\leq p_{i}-1$. We also know that
$\hat{\Delta}=\mathbb{Z}^{\ast}\delta\cup\{\Delta+\mathbb{Z}\delta\}$.

From Proposition \ref{prop Grothendieck group}, there is a natural
isomorphism of $\mathbb{Z}$-modules
$K_{0}(\Coh(\mathbb{X}))\cong\hat{Q}$ given by
\begin{align*}
&[S^{i}_{j}]\mapsto\alpha_{ij},\quad \text{for}\ 1\leq j\leq p_{i}-1, 1\leq i\leq n,\\
&[S^{i}_{p_{i}}]\mapsto\delta-\sum_{j=1}^{p_{i}-1}\alpha_{ij},\quad \text{for}\ 1\leq i\leq n,\\
&[\mo(k\vec{c})]\mapsto\alpha_{\ast}+k\delta,\quad \text{for}\ k\in\mathbb{Z}.
\end{align*}

Now by Lemma \ref{lem comute Euler form} we have the following result:
\begin{Lem}
The symmetric Euler form on $K_{0}(\Coh(\mathbb{X}))$ coincides with
the Cartan form on $\hat{Q}$.
\end{Lem}

\subsection{}\label{subsec Schiffmann's result}
In \cite{S} (also see \cite{S0}) Schiffmann has proved that the Hall
algebra $\mathbf{H}(\Coh(\mathbb{X}))$ provides a realization of the
quantized enveloping algebra of a certain nilpotent subalgebra of
$\mathcal{L}\mathfrak{g}$, denoted by
$U_{v}(\widehat{\mathfrak{n}})$.

Let us recall the definition of $U_{v}(\widehat{\mathfrak{n}})$.
First, for each $i$, we denote by $U_{i}$ the subalgebra of
$U_{v}(\mathcal{L}\mathfrak{g})$ generated by $x_{[i,j],k}^{\pm}$,
$h_{[i,j],l}$ and $K_{[i,j]}^{\pm 1}$ ($1\leq j\leq p_{i}-1$,
$k\in\mathbb{Z}$, $l\in\mathbb{Z}\setminus\{0\}$). It is clear
that this subalgebra is isomorphic to
$U_{v}(\widehat{\mathfrak{sl}}_{p_{i}})$. Denote the Chevalley
generators of $U_{v}(\widehat{\mathfrak{sl}}_{p_{i}})$ by
$E^{i}_{j}$ and $F^{i}_{j}$ ($1\leq j\leq p_{i}$). Then the image of $E^{i}_{j}$ under Beck's
isomorphism is $x^{+}_{[i,j],0}$, for $1\leq
j\leq p_{i}-1$. But the image of $E^{i}_{p_{i}}$, which we denote by $\varepsilon_{i}$, is not a Drinfeld
generator. Now let $U_{i}^{+}$ be the subalgebra
generated by $x^{+}_{[i,j],0}$ and $\varepsilon_{i}$. Thus
$U_{i}^{+}$ is isomorphic to the standard positive part of
$U_{v}(\widehat{\mathfrak{sl}}_{p_{i}})$. Finally,
$U_{v}(\widehat{\mathfrak{n}})$ is defined to be the subalgebra of
$U_{v}(\mathcal{L}\mathfrak{g})$ generated by $x_{\ast,k}^{+}$,
$h_{\ast,r}$ and $U_{i}^{+}$ ($k\in\mathbb{Z}$, $r\geq 1$, $1\leq
i\leq n$).

Now let $\mathbf{C}(\Coh(\mathbb{X}))$ be the subalgebra of
$\mathbf{H}(\Coh(\mathbb{X}))$ generated by
$u_{\mathscr{O}(k\vec{c})}$, $u_{S^{i}_{j}}$ and $T_{r}$ for
$k\in\mathbb{Z}$, $1\leq i\leq n$, $1\leq j\leq p_{i}$ and
$r\in\mathbb{N}$. It is called the \textit{composition algebra} of $\Coh(\mathbb{X})$.
The following theorem was proved by Schiffmann:

\begin{Thm}[\cite{S}]\label{thm Schiffmann}
The assignment $x^{+}_{[i,j],0}\mapsto u_{S^{i}_{j}}$ for $1\leq
j\leq p_{i}-1$, $\varepsilon_{i}\mapsto u_{S^{i}_{p_{i}}}$,
$x_{\ast,k}\mapsto u_{\mathscr{O}(k\vec{c})}$, $h_{\ast,r}\mapsto
T_{r}$ gives an epimomorphism of algebras
$$\Phi: U_{v}(\widehat{\mathfrak{n}})\twoheadrightarrow\mathbf{C}(\Coh(\mathbb{X})).$$
Moreover, if $\mathfrak{g}$ is of finite or affine type, $\Phi$ is
an isomorphism.
\end{Thm}

The subalgebra $U_{v}(\widehat{\mathfrak{n}})$ can be viewed as a certain ``positive part" of
$U_{v}(\mathcal{L}\mathfrak{g})$. But by definition it is not generated by part of Drinfeld generators. Thus the correspondence between
generators of $\mathbf{C}(\Coh(\mathbb{X}))$ and Drinfeld generators of $U_{v}(\mathcal{L}\mathfrak{g})$ is not completely
explicit.

As in \ref{subsec double Hall alg}, we have the double Hall algebra $\mathbf{DH}(\Coh(\mathbb{X}))$. We define the \textit{double composition algebra} $\mathbf{DC}(\Coh(\mathbb{X}))$ to be the subalgebra of
$\mathbf{DH}(\Coh(\mathbb{X}))$ generated by $\mathbf{C}(\Coh(\mathbb{X}))$,
$\mathbf{C}^{-}(\Coh(\mathbb{X}))$ (the subalgebra of $\mathbf{H}^{-}(\Coh(\mathbb{X}))$ defined similar to $\mathbf{C}(\Coh(\mathbb{X}))$) and the torus $\mathbf{T}$. Recall that $\mathbf{T}=\{K_{\alpha}|\alpha\in K_{0}(\Coh(\mathbb{X}))\}$. The following notations will be used: $K_{\ast}=K_{[\mathscr{O}]}$, $K_{[i,j]}=K_{[S_{j}^{i}]}$ for $1\leq i\leq n$, $1\leq j\leq p_{i}-1$.

We will show that the Drinfeld generators and relations for the whole quantum loop
algebra can be fully understood in the double composition algebra.

\subsection{}\label{subsec main theorem}
We keep the notations in the previous subsections.

Note that in
\ref{subsec element T_r} and \ref{subsec more about non-homogeneous
tube} we have defined the elements $T_{k}$, $\pi^{i}_{j,k}$ in the
Hall algebra $\mathbf{H}(\Coh(\mathbb{X}))$ for any $1\leq i\leq n$,
$1\leq j\leq p_{i}-1$ and $k\geq 1$. Similarly we can define the elements $T^{-}_{k}$, $\pi^{-i}_{j,k}$
in the negative Hall algebra $\mathbf{H}^{-}(\Coh(\mathbb{X}))$.

Moreover, we define
$$\eta^{+}_{i,j}=v^{1-j}\Theta_{\lambda_{i}}(\sum_{M\in\mathcal{M}_{j+1,\delta-e_j}}(1-v^2)^{dim End(M)-1}u^+_M)K_{[i,j]},$$
$$\eta^{-}_{i,j}=-v^{-j}\Theta^{-}_{\lambda_{i}}(\sum_{M\in\mathcal{M}_{j+1,\delta-e_j}}(1-v^2)^{dim
End(M)-1}u^-_M)K^-_{[i,j]}.$$

The following theorem is the main result of this paper:
\begin{Thm}\label{thm main}
For any star-shaped graph $\Gamma$, let $\mathfrak{g}$ be the
Kac-Moody algebra and $\mathbb{X}$ be the weighted projective line
associate to $\Gamma$ respectively. Then the following elements together with $\{K_{s}\in\mathbf{T}|s\in\Gamma_{0}\}$ in
the double Hall algebra $\mathbf{DH}(\Coh(\mathbb{X}))$ satisfy the
defining relations of the quantum loop algebra $U_{v}(\mathcal{L}\mathfrak{g})$ (see
\ref{subsec quantum loop alg}).

\begin{equation*}
h_{s,r}=\begin{cases}
        T_{r}&s=\ast,\ r>0\\
        -T^{-}_{-r}&s=\ast,\ r<0\\
        \pi^{i}_{j+1,r}-(v^{r}+v^{-r})\pi^{i}_{j,r}+\pi^{i}_{j-1,r}&s=[i,j],\
        r>0\\
        -\pi^{-i}_{j+1,-r}+(v^{r}+v^{-r})\pi^{-i}_{j,-r}-\pi^{-i}_{j-1,-r}&s=[i,j],\
        r<0\\
        \end{cases}
\end{equation*}

\begin{equation*}
x^{+}_{s,t}=\begin{cases}
            u^{+}_{\mathscr{O}(t\vec{c})}& s=\ast,\ t\in\mathbb{Z}\\
            u^{+}_{S^{i}_{j}}& s=[i,j],\ t=0\\
            \frac{t}{[2t]}[h_{[i,j],t},x^+_{[i,j],0}] & s=[i,j],\ t\geq 1\\
            \eta^{-}_{i,j}& s=[i,j],\ t=-1\\
            \frac{-k}{[-2k]}[h_{[i,j],-k},x^+_{[i,j],-1}]& s=[i,j],\ t=-k-1,\ k>0
            \end{cases}
\end{equation*}

\begin{equation*}
x^{-}_{s,t}=\begin{cases}
            -vu^{-}_{\mathscr{O}(-t\vec{c})}&s=\ast,\ t\in\mathbb{Z}\\
            \eta^{+}_{i,j}&s=[i,j],\ t=1\\
            \frac{-k}{[2k]}[h_{[i,j],k},x^-_{[i,j],1}]&s=[i,j],\ t=k+1,\ k>0\\
            -vu^{-}_{S^{i}_{j}}&s=[i,j],\ t=0\\
            \frac{k}{[-2k]}[h_{[i,j],-k},x^-_{[i,j],0}]&s=[i,j],\ t=-k,\ k>0
            \end{cases}
\end{equation*}
\end{Thm}

The proof of this theorem consists of the next three sections.

\begin{remark}
(1). By the above theorem we have an algebra homomorphism
$$\Xi:U_{v}(\mathcal{L}\mathfrak{g})\rightarrow\mathbf{DH}(\Coh(\mathbb{X})).$$

The image of $\Xi$ is in fact the double composition algebra $\mathbf{DC}(\Coh(\mathbb{X}))$. This can be easily seen from the following results proved in \cite{H}:
$$u_{S^{i}_{p_{i}}}=(-1)^{n}[-vu^{-}_{S^{i}_{p_{i}-1}},\cdots,[-vu^{-}_{S^{i}_{2}},\eta_{i,1}^{+}]_{v^{-1}}]_{v^{-1}}K_{[S_{p_{i}}^{i}]},$$
where the $v$-commutator is defined as in Lemma \ref{lem T_omega}. And for any $i$, the elements $\pi^{i}_{j+1,r}-(v^{r}+v^{-r})\pi^{i}_{j,r}+\pi^{i}_{j-1,r}$ and $\eta^{+}_{i,j}$ are in $\Theta_{i}(\mathbf{C}(C_{p_{i}}))\subset\mathbf{C}(\Coh(\mathbb{X}))$.

(2). Shortly after the first version of this paper appeared in arXiv, Burban and Schiffmann proved in \cite{BS2} that the composition algebra $\mathbf{C}(\Coh(\mathbb{X}))$ is a topological sub-bialgebra of $\mathbf{H}(\Coh(\mathbb{X}))$. Thus we can construct the reduced Drinfeld double of the algebra $\mathbf{C}(\Coh(\mathbb{X}))$. Now it is clear that our double composition algebra $\mathbf{DC}(\Coh(\mathbb{X}))$ coincides with the reduced Drinfeld double of $\mathbf{C}(\Coh(\mathbb{X}))$. However, we do not need this result throughout the paper.

(3). We expect that our homomorphism $\Xi$ is injective, namely $\Xi:U_{v}(\mathcal{L}\mathfrak{g})\simeq\mathbf{DC}(\Coh(\mathbb{X}))$, at least in the case that $\mathfrak{g}$ is of finite or affine type. Note that it was shown in \cite{BS2} that the two algebras $U_{v}(\mathcal{L}\mathfrak{g})$ and $\mathbf{DC}(\Coh(\mathbb{X}))$ are isomorphic in the finite type case, where the isomorphism is induced by a derived equivalence functor of the category $\Coh(\mathbb{X})$ and representations of a certain tame quiver. But it is difficult to understand Drinfeld's presentation by such isomorphisms. In particular, even if our $\Xi$ is an isomorphism, it cannot be the isomorphism induced by a derived equivalence functor for most cases. We will give a more detailed explanation in Section 9.
\end{remark}

\section{Relations in the subalgebras isomorphic to $U_{v}(\widehat{\mathfrak{sl}}_{p_{i}})$}\label{sec relations in a tube}

\subsection{Relations in each tube $\mathscr{T}_{\lambda_{i}}$}
Recall the star-shaped graph $\Gamma$ in \ref{subsec
star-shaped graph}. We can see that for any fixed
$i\in\{1,2,\cdots,n\}$, the full subgraph consisting of vertices
$\{[i,j]|1\leq j\leq p_{i}-1\}$ is a Dynkin diagram of type
$A_{p_{i}-1}$. Thus the relations to be satisfied by the elements
$x^{\pm}_{[i,j],k}$, $h_{[i,j],r}$ for all $1\leq j\leq p_{i}-1$,
$k\in\mathbb{Z}$, $r\in\mathbb{Z}\setminus\{0\}$ are actually the
defining relations of $U_{v}(\widehat{\mathfrak{sl}}_{p_{i}})$. We
will prove them in this section.

Note that by definition the elements $x^{+}_{[i,j],k}$,
$x^{-}_{[i,j],l}$, $h_{[i,j],r}$ for $1\leq\ j\leq p_{i}-1$,
$k\in\mathbb{N}$, $l,r\in\mathbb{N}^{\ast}$ are all in the
subalgebra $\mathbf{H}(\mathscr{T}_{\lambda_{i}})$, which is
isomorphic to $\mathbf{H}(C_{p_{i}})$. Thus we can use the method
developed by Hubery in \cite{H}, where he explicitly write down the
elements in $\mathbf{H}C_{m}$ satisfying Drinfeld relations of
$U^{+}_{v}(\widehat{\mathfrak{sl}}_{m})$. Then by the isomorphism
$\Theta_{\lambda_{i}}$, we can transfer the result to
$\mathbf{H}(\mathscr{T}_{\lambda_{i}})\subset\mathbf{H}(\Coh(\mathbb{X}))$.
Namely we have:

\begin{Prop}[\cite{H}]\label{prop Hubery's result}
For any fixed $i\in\{1,2,\cdots,n\}$, the elements
$x^{+}_{[i,j],k}$, $x^{-}_{[i,j],l}$, $h_{[i,j],r}$ for $1\leq\
j\leq p_{i}-1$, $k\in\mathbb{N}$, $l,r\in\mathbb{N}^{\ast}$ satisfy
the Drinfeld relations of
$U^{+}_{v}(\widehat{\mathfrak{sl}}_{p_{i}})$.
\end{Prop}

This result can be easily extended to $U_{v}(\widehat{\mathfrak{sl}}_{p_{i}})$:

\begin{Cor}\label{Cor relations in a tube}
For any fixed $i\in\{1,2,\cdots,n\}$, the elements
$x^{\pm}_{[i,j],k}$, $h_{[i,j],r}$ for all $1\leq j\leq p_{i}-1$,
$k\in\mathbb{Z}$, $r\in\mathbb{Z}\setminus\{0\}$ satisfy the
Drinfeld relations for $U_{v}(\widehat{\mathfrak{sl}}_{p_{i}})$.
\end{Cor}

The proof will be given in \ref{subsec Hubery's arguments}.

\subsection{Beck's isomorphism for $U_{v}(\widehat{\mathfrak{sl}}_{m})$}
In this subsection we briefly recall Beck's isomorphism in
\cite{Be}.

Let $W$ be the Weyl group of $\mathfrak{sl}_{m}$. The simple reflections
$s_1,\ldots, s_{m-1}$ generates $W$. Let $P$ be the weight
lattice and $Q$ be the root lattice. The fundamental weights are
denoted by $\omega_1,\ldots,\omega_{m-1}$.

The \emph{extended affine Weyl group} is defined to be the
semi-direct product $\widetilde{W}=P\rtimes W$, where $(x,\omega)(x',\omega')=(x+\omega(x'),\omega\omega')$.
And the \emph{affine Weyl group} associated to
$\widehat{\mathfrak{sl}}_{m}$ is the subgroup $\widehat{W}=Q\rtimes
W$. We have the decomposition $\widetilde{W}=\widehat{W}\rtimes
(\mathbb{Z}/{m\mathbb{Z}})$, where the cyclic group
$\mathbb{Z}/{m\mathbb{Z}}$ has a generator $\tau=
(\omega_1,s_1s_2\cdots s_{m-1})$.

Set $s_m:=(\omega_1+\omega_{m-1}, s_1s_2\cdots s_{m-1}\cdots s_2s_1)$.
Then $\{s_1, s_2,\ldots, s_m\}$ is a set of generators of the affine
Weyl group $\widehat{W}$. We can extend the length function on
$\widehat{W}$ to $\widetilde{W}$ by setting $l(\tau)=0$.

Note that the fundamental weights, considered as elements
in $\widetilde{W}$, have the following reduced expressions in
terms of the generators $s_i$ and $\tau$:
$$\omega_i=\tau^i(s_{m-i}\cdots s_{m-1})\cdots(s_2\cdots
s_{i+1})(s_1\cdots s_i).$$

The braid group associated to $\widetilde{W}$ is the group with
generators $T_\omega$ ($\omega\in \widetilde{W}$) the relations
$T_\omega T_\omega'=T_{\omega\omega'}$ if
$l(\omega)+l(\omega')=l(\omega\omega')$. Following Lusztig, it acts
on $U_v(\widehat{sl}_m)$ (see \ref{subsec quantum group}) via the following rules:
\begin{gather*}
T_i(E_i)=-F_iK_i,\ \ T_i(F_i)=-K_i^{-1}E_i,\ \ T_i(K_\alpha)=K_{s_i(\alpha)}, \\
T_i(E_j)=\sum_{r+s=-c_{ij}}(-1)^rv^{-r}E_i^{(s)}E_jE_i^{(r)},\ for\
i\neq j, \\
T_i(F_j)=\sum_{r+s=-c_{ij}}(-1)^rv^rF_i^{(r)}F_jF_i^{(s)},\ for\
i\neq j \\
T_{\tau}(K_i)=K_{i+1},\ \ T_\tau(E_i)=E_{i+1},\ \
T_\tau(F_i)=F_{i+1}.
\end{gather*}
where $E_{i}^{(r)}=E_{i}^r/{[r]!}$,
$F_{i}^{(r)}=F_{i}^r/{[r]!}$ and $(c_{ij})_{1\leq i,j\leq m}$ is the Cartan matrix
associated to $\widehat{sl}_m$.

For $1\leq i\leq m-1$ and $j\in \mathbb{Z}$, let
$$x^-_{i,j}=(-1)^{ij}v^{mj}T^j_{\omega_i}(F_i),\ \ x^+_{i,j}=(-1)^{ij}v^{mj}T^{-j}_{\omega_i}(E_i).$$

For $1\leq i\leq m-1$, $k>0$, define $h_{i,k}$ via the following
generating functions
$$K_i\exp((v-v^{-1})\Sigma_{k>0}h_{i,k}u^k )=
 \Sigma_{l\geq 0}\psi_{i,l}u^{l},$$
where $\psi_{i,l}=(v-v^{-1})[E_i,T^l_{\omega_i}(F_i)]$ for $l>0$ and $\psi_{i,0}=K_i$.

Similarly, define $h_{i,-k}$ via
$$K^{-1}_i\exp((v^{-1}-v)\Sigma_{k>0}h_{i,-k}u^{k}) =
 \Sigma_{l\geq 0}\psi_{i,-l}u^{l}$$
where $\varphi_{i,-l}=(v-v^{-1})[F_i,T^l_{\omega_i}(E_i)]$ for $l>0$ and $\varphi_{i,0}=K^{-1}_i$.

 The following is now well-known, see \cite{Be} for the proof.

\begin{Thm}\label{thm Beck's iso}
$U_v(\widehat{sl}_m)$ is generated by the elements
$x^{\pm}_{i,j},h_{i,k}, K^{\pm1}_i$, where $1\leq i\leq m-1$, $j\in
\mathbb{Z}$ and $k\in \mathbb{Z}\backslash \{0\}$. The defining
relations are Drinfeld relations for
$U_{v}(\mathcal{L}\mathfrak{sl}_{m})$. Thus the Drinfeld-Jimbo
presentation of $U_v(\widehat{sl}_m)$ is isomorphic to the Drinfeld
presentation of $U_{v}(\mathcal{L}\mathfrak{sl}_{m})$.
\end{Thm}

\subsection{The proof of Prop \ref{prop Hubery's result} and Cor \ref{Cor relations in a tube}}
\label{subsec Hubery's arguments}
Proposition \ref{prop Hubery's result} is a result of Hubery
\cite{H}. Corollary \ref{Cor relations in a tube} follows easily
by a similar method.

Now we briefly recall the arguments in \cite{H}. Let $C_{m}$ be the cyclic quiver with $m$ vertices and
consider the Hall algebra $\mathbf{H}(C_{m})$. The composition
algebra $\mathbf{C}(C_{m})$ is the subalgebra of $\mathbf{H}(C_{m})$ generated by $u_{S_{i}}$ for $1\leq i\leq m$.
We know that the composition subalgebra $\mathbf{C}(C_{m})$ is
isomorphic to $U^{+}_{v}(\widehat{\mathfrak{sl}}_{m})$ (Theorem
\ref{thm Ringel-Green}) and the reduced Drinfeld double
$\mathbf{DC}(C_{m})$ is isomorphic to
$U_{v}(\widehat{\mathfrak{sl}}_{m})$ (Theorem \ref{thm Xiao}).
And the isomorphism is given by
$$u_{S_{i}}\mapsto E_{i},\ \ -vu^{-}_{S_{i}}\mapsto F_{i}.$$

We also know that $U_{v}(\widehat{\mathfrak{sl}}_{m})$ is isomorphic
to $U_{v}(\mathcal{L}\mathfrak{sl}_{m})$. Let
$$\Upsilon:\mathbf{DC}(C_{m})\simeq U_{v}(\mathcal{L}\mathfrak{sl}_{m})$$
be the composition of two isomorphisms mentioned above.

Now if we can find the inverse images of the elements
$x^{\pm}_{i,j},h_{i,k}$ under $\Upsilon$, they should certainly
satisfy the Drinfeld relations.

By Theorem \ref{thm Beck's iso} we have
$$x^-_{i,1}=(-1)^iv^mT_{\omega_i}(F_i),\ \
x^+_{i,-1}=(-1)^{-i}v^{-m}T_{\omega_i}(E_i)$$

Recall that $\omega_i=\tau^i(s_{n-i}\cdots s_{n-1})\cdots(s_2\cdots
s_{i+1})(s_1\cdots s_i)$. By induction, we have the following
result:
\begin{Lem}\label{lem T_omega}
For $1\leq i\leq m-1$, we have
$$T_{\omega_i}(F_i)=-K_i[E_m,
E_{m-1},\ldots,E_{i+1},E_1,E_2,\ldots,E_{i-1}]_{v^{-1}},$$
$$T_{\omega_i}(E_i)=-[F_{i-1},\ldots F_2,F_1,F_{i+1},\ldots,
F_{m-1},F_m]_{v}K^{-1}_i,$$
where $[a,b]_v=ab-vba$, and $[a,b,c]_{v}=[[a,b]_{v},c]_{v}$.
\end{Lem}

We identify $E_i$ (resp. $F_i$) with $u^+_{s_i}$ (resp.
$-vu^-_{s_i}$). Further calculations yield
$$T_{\omega_i}(F_i)=-v^{-m+i+1}K_i[u^+_{s_{m}(m-i)},
u^+_{s_1},u^+_{s_2}\ldots,u^+_{s_{i-1}}]_{v^{-1}},$$
$$T_{\omega_i}(F_i)K^{-1}_i=(-1)^iv^{1-m-i}\sum_{M\in\mathcal{M}_{i+1,\delta-e_j}}(1-v^2)^{dim
End(M)-1}u^+_M.$$

Therefore, we have
\begin{align*}
x^+_{i,0}&=u^+_{S_i},\\
x^-_{i,1}&=v^{1-i}\sum_{M\in\mathcal{M}_{i+1,\delta-e_i}}(1-v^2)^{dim
End(M)-1}u^+_MK_i.
\end{align*}

In a similar way, we have
\begin{align*}
x^-_{i,0}&=u^-_{S_i},\\
x^+_{i,-1}&=-v^{-1}v^{1-i}\sum_{M\in\mathcal{M}_{i+1,\delta-e_i}}(1-v^2)^{dim
End(M)-1}u^-_MK^-_i.
\end{align*}

The inverse image of the elements $h_{i,k}$ in the Hall algebra can be
found by induction using the fact that $\pi_{n,r}$ (see \ref{subsec
more about non-homogeneous tube}) is central and primitive in the Hall
algebra. And the result is:
$$h_{i,k}=\pi_{i+1,k}-(v^k+v^{-k})\pi_{i,k}+\pi_{i-1,k},\ \ for\ k>0.$$

Again by a similar method we have
$$h_{i,-k}=-(\pi^-_{i+1,k}-(v^k+v^{-k})\pi^-_{i,k}+\pi^-_{i-1,k}),\ \ for\ k>0.$$

Finally, the elements $x^{+}_{i,j}$ ($j\neq 0,-1$) and
$x^{-}_{i,j}$ ($j\neq 0,1$) are determined by the relation \ref{subsec
quantum loop alg} (4).

\section{Relations in $\mathbf{H}(\Coh(\mathbb{X}))$}\label{sec relation of positive part}
In this section we focus on the elements $x^{+}_{[i,j],k}$, $x^{-}_{[i,j],k+1}$
$h_{[i,j],l}$, $x^{+}_{\ast,r}$, and $h_{\ast,t}$ where $1\leq i\leq
n$, $1\leq j\leq p_{i}-1$, $k\geq 0$, $l,t\in\mathbb{N}$ and
$r\in\mathbb{Z}$. These are the elements belonging to the positive Hall algebra $\mathbf{H}(\Coh(\mathbb{X}))$.

\subsection{The known relations}\label{subsec known relations}
First, for reader's convenience, we list the relations which has already been
proved in \cite{S}:

(a). For $1\leq i\leq n$, $r, r_1\in\mathbb{N}$ and $k\in\mathbb{Z}$,
$$[T_{*,r},
u_{\mathscr{O}(k\vec{c})}]=\frac{[2r]}{r}u_{\mathscr{O}((k+r)\vec{c})},$$
$$[T_{*,r},x^+_{[i,1],r_1}]=-\frac{[r]}{r}x^+_{[i,1],r_1+r}.$$

(b). For $t_1, t_2\in\mathbb{Z}$,
\begin{equation*}
u_{\mathscr{O}((t_1+1)\vec{c})}u_{\mathscr{O}(t_2\vec{c})}
-v^2u_{\mathscr{O}(t_2\vec{c})}u_{\mathscr{O}((t_1+1)\vec{c})}=
v^{2}u_{\mathscr{O}(t_1\vec{c})}u_{\mathscr{O}((t_2+1)\vec{c})}-
u_{\mathscr{O}((t_2+1)\vec{c})}u_{\mathscr{O}(t_1\vec{c})}.
\end{equation*}

(c). For $1\leq i\leq n$, $r, r_1, r_2\in\mathbb{N}$ and $t, t_1, t_2\in\mathbb{Z}$,
\begin{equation*}
\Sym_{r_1,r_2}\{x^+_{[i,1],r_1}x^+_{[i,1],r_2}u_{\mathscr{O}(t\vec{c})}
-[2]x^+_{[i,1],r_1}u_{\mathscr{O}(t\vec{c})}x^+_{[i,1],r_2}
+u_{\mathscr{O}(t\vec{c})}
x^+_{[i,1],r_1}x^+_{[i,1],r_2}\}=0,
\end{equation*}
\begin{equation*}
\Sym_{t_1,t_2}\{u_{\mathscr{O}(t_1\vec{c})}u_{\mathscr{O}(t_2\vec{c})}x^+_{[i,1],r}
-[2]u_{\mathscr{O}(t_1\vec{c})}x^+_{[i,1],r}u_{\mathscr{O}(t_2\vec{c})}
+x^+_{[i,1],r}u_{\mathscr{O}(t_1\vec{c})}u_{\mathscr{O}(t_2\vec{c})}\}
=0.
\end{equation*}

(d). For $1\leq i\leq n$, $r, r_1, r_2\in\mathbb{N}$ and $t\in\mathbb{Z}$,
$$u_{\mathscr{O}((t+1)\vec{c})}x^+_{[i,1],r}-v^{-1}x^+_{[i,1],r}u_{\mathscr{O}((t+1)\vec{c})}
=v^{-1}u_{\mathscr{O}(t\vec{c})}x^+_{[i,1],r+1}-x^+_{[i,1],r+1}u_{\mathscr{O}(t\vec{c})},$$
$$x^+_{[i,1],r_1+1}x^+_{[i,1],r_2}-v^2x^+_{[i,1],r_2}x^+_{[i,1],r_1+1}=v^2x^+_{[i,1],r_2+1}x^+_{[i,1],r_1}-x^+_{[i,1],r_1}x^+_{[i,1],r_2+1}.$$

For any two elements arising from different tubes, we know that they commute with each other as
there are no non-trivial extensions between torsion sheaves belonging to different tubes.

Moreover, from the last section we know that the elements
in each $\mathbf{H}(\mathscr{T}_{\lambda_{i}})$ ($1\leq i\leq n$) satisfy
the required relations.

\subsection{The remaining relations}
\begin{Lem}
(1). $[h_{[i,j],k},h_{[i,l],m}]=0$, for any $1\leq i\leq n,1\leq j,l\leq p_i-1,k,m\in\mathbb{N}$.

(2). $[h_{[i,j],l},h_{\ast,k}]=0$, for any $1\leq i\leq n,1\leq j\leq p_i-1,l,k\in\mathbb{N}$.
\end{Lem}

\begin{proof}
We deduce $[\pi^i_{j,k},\pi^i_{l,m}]=0$ by the embedding of algebras
$$\mathbf{H}_v(C_1)\hookrightarrow
\mathbf{H}_v(C_2)\cdots\hookrightarrow\mathbf{H}_v(C_{p_{i}})$$ and
the fact that $\pi^i_{j,k}$ is in the center of $\mathbf{H}_v(C_j)$
(see the remark in \ref{subsec more about non-homogeneous tube}).
The first equation follows.

For the second one, we have
$[h_{[i,j],l},h_{*,k}]=[h_{[i,j],l},\pi^i_{1,k}]=0$.
\end{proof}

\begin{Lem}\label{lem hsi}
$[h_{*,l},x^+_{[i,j],k}]=0$, for any $1\leq i\leq n, 2\leq j\leq p_{i}-1$, $l,k\in\mathbb{N}$.
\end{Lem}

\begin{proof}
First we prove the case $k=0$, we know that
$x^+_{[i,j],0}=u_{S^i_j}$. Assume that $u_{\mathscr{G}}$ is a term with non-zero coefficient in the expression of
$h_{*,l}=T_{l}$, where $\mathscr{G}$ is a torsion sheaf. We only need to consider the direct summand $\mathscr{F}$ of $\mathscr{G}$
belonging to $\mathscr{T}_{\lambda_{i}}$. We have
$[\mathscr{F}]=l\delta$ and $\Top(\mathscr{F})= S^i_0\oplus S^i_0
\cdots \oplus S^i_0$, $\soc(\mathscr{F})=S^i_1\oplus
S^i_1\cdots\oplus S^i_1$. It is clear that
$[u_{\mathscr{F}},u_{S^i_j}]=0$. Thus we get
$[h_{*,l},x^+_{[i,j],0}]=0$.

For the general case, one just need to apply $\ad(h_{[i,j],k})$ to
the above formula.
\end{proof}

\begin{Lem}\label{lem hsi--}
$[h_{*,l},x^-_{[i,1],k}]=\frac{[l]}{l}x^-_{[i,1],1+k}$, for any
$l\in\mathbb{N}$,$1\leq i\leq n$ and $k\geq 1$.
\end{Lem}

\begin{proof}
Again we first consider the simplest case, namely the case $k=1$.

We know that $x^-_{[i,1],1}=u_{S^i_0(p_i-1)}K_{[i,1]}$. It is not difficult to see
that $[\pi^i_{2,k}, x^-_{[i,1],1}]=0$. Hence we have
$$[h_{*,l},x^-_{[i,1],1}]=[\mathbf{h}_{l,\lambda_{i}},x^-_{[i,1],1}]
=\frac{1}{-(v^l+v^{-l})}[h_{[i,1],l},x^-_{[i,1],1}]
=\frac{[l]}{l}x^-_{[i,1],l+1}.$$

For $k>1$, we have
\begin{equation*}
\begin{split}
[h_{*,l},x^-_{[i,1],k}]&=\frac{-(k-1)}{[2(k-1)]}[h_{*,l},[h_{[i,1],k-1},x^-_{[i,1],1}]\\
&=\frac{-(k-1)}{[2(k-1)]}[h_{[i,1],k-1},[h_{*,l},x^-_{[i,1],1}]\\
&=\frac{-(k-1)}{[2(k-1)]}\frac{[l]}{l}[h_{[i,1],k-1},x^-_{[i,1],l+1}]
=\frac{[l]}{l}x^-_{[i,1],k+l}.
\end{split}
\end{equation*}
\end{proof}

\begin{Lem}\label{lem h_star x- jgeq2}
$[h_{*,l},x^-_{[i,j],k}]=0$, for any $1\leq i\leq n, 2\leq j\leq p_i-1$, $l\in\mathbb{N}$ and
$k\geq 1$.
\end{Lem}

\begin{proof}
The case $k=1$ can be proved in the same way as Lemma \ref{lem hsi}.
For $k>1$ we just apply $\ad(h_{[i,j],k})$.
\end{proof}

\begin{Lem}\label{lem hsi-}
$[x^+_{*,k},x^-_{[i,j],1}]=0$, for any $k\in\mathbb{Z}$,$1\leq i\leq n,1\leq j\leq p_i-1$.
\end{Lem}

\begin{proof}
We first consider the case $j=1$. Note that
$x^-_{[i,1],1}=u_{S^i_0(p_i-1)}K_{[i,1]}$, so we have
\begin{equation*}
\begin{split}
[x^+_{*,k},u_{S^i_0(p_i-1)}K_{[i,1]}]&=
x^+_{*,k}u_{S^i_0(p_i-1)}K_{[i,1]}-u_{S^i_0(p_i-1)}K_{[i,1]}x^+_{*,r}\\
&=(x^+_{*,k}u_{S^i_0(p_i-1)}-v^{-1}u_{S^i_0(p_i-1)}x^+_{*,r})K_{[i,1]}.
\end{split}
\end{equation*}

We know that
$$\Hom(\mathscr{O}(k\vec{c}),S^i_0(p_i-1))=k,$$
$$\Ext^{1}(\mathscr{O}(k\vec{c}),S^i_0(p_i-1))=\Hom(S^i_0(p_i-1),\mathscr{O}(k\vec{c}))=0.$$

Thus we have
$$x^+_{*,k}u_{S^i_0(p_i-1)}=vu_{\mathscr{O}(k\vec{c})\oplus S^i_0(p_i-1)},\ \
u_{S^i_0(p_i-1)}x^+_{*,k}=v^2u_{\mathscr{O}(k\vec{c})\oplus
S^i_0(p_i-1)}.$$

Then it follows that $[x^+_{*,k},u_{S^i_0(p_i-1)}K_{[i,1]}]=0$.

Now assume $j\geq 2$, in this case we know that
$$x^-_{[i,j],1}=(-v)^{j+1}K_{[i,j]}[u_{S^i_0(p_{i}-j)},u_{S^i_1},\ldots,
u_{S^i_{j-1}} ]_{v^{-1}}.$$

It suffices to prove
$[x^+_{*,k},[u_{S^i_0(p_i-j)},u_{S^i_1}]_{v^{-1}}]=0$, which can be
deduced using the following identities:
$$u_{S^i_0(p_i-j)}u_{\mathscr{O}(k\vec{c})}=vu_{\mathscr{O}(k\vec{c})}u_{S^i_0(p_i-j)},$$
$$u_{S^i_1}u_{\mathscr{O}(k\vec{c})}=v^{-1}(u_{\mathscr{O}(k\vec{c})}u_{S^i_1}
 +u_{\mathscr{O}(k\vec{c}+\vec{x}_{i})}),$$
$$u_{S^i_0(p_i-j)}u_{\mathscr{O}(k\vec{c}+\vec{x}_i)}=
u_{\mathscr{O}(k\vec{c}+\vec{x}_i)}u_{S^i_0(p_i-j)}.$$
\end{proof}

\begin{Lem}\label{lem h_i1 x^+_ast}
$[h_{[i,1],l},x^+_{*,k}]=\frac{-[l]}{l}x^+_{*,k+l}$, for any $1\leq i\leq n$,
$l\in\mathbb{N}$ and $k\in\mathbb{Z}$
\end{Lem}

\begin{proof}
For fixed $i$, we know the following relation holds (see
Corollary \ref{Cor relations in a tube}):
$$[x^{+}_{[i,j],k},x^{-}_{[i,j],l}]=\frac{\psi_{[i,j],k+l}-\varphi_{[i,j],k+l}}{v-v^{-1}}.$$

Now we set
$$\xi^i_r=[x^+_{[i,1],r-1},x^-_{[i,1],1}]K^{-1}_{[i,1]}=\frac{1}{v-v^{-1}}\psi_{[i,1],r}K^{-1}_{[i,1]}.$$

By the definition of $\psi_{[i,1],r}$ we have
$$rh_{[i,1],r}=r\xi^i_r-\sum^{r-1}_{s=1}(v-v^{-1})sh_{[i,1],s}\xi^i_{r-s}.\eqno(*1)$$

The definition of $\xi^{i}_{r}$ yields
$$\xi^i_r=x^+_{[i,1],r-1}u_{S^i_0(p_i-1)}-v^2u_{S^i_0(p_i-1)}x^+_{[i,1],r-1}.$$

Thus we have
\begin{equation*}
\begin{split}
\xi^i_ru_{\mathscr{O}(k\vec{c})}&=
x^+_{[i,1],r-1}u_{S^i_0(p_i-1)}u_{\mathscr{O}(k\vec{c})}
-v^2u_{S^i_0(p_i-1)}x^+_{[i,1],r-1}u_{\mathscr{O}(k\vec{c})}\\
& =vx^+_{[i,1],r-1}u_{\mathscr{O}(k\vec{c})}u_{S^i_0(p_i-1)}
-v^2u_{S^i_0(p_i-1)}x^+_{[i,1],r-1}u_{\mathscr{O}(k\vec{c})}.
\end{split}
\end{equation*}

We claim that the following identity holds:
\begin{equation*}
\begin{split}
\xi^i_ru_{\mathscr{O}(k\vec{c})}&=u_{\mathscr{O}(k\vec{c})}
\xi^i_r+(v^{-1}-v)u_{\mathscr{O}((k+1)\vec{c})}\xi^i_{r-1}
+v^{-1}(v^{-1}-v)u_{\mathscr{O}((k+2)\vec{c})}\xi^i_{r-2}+\\&\cdots
+v^{-(r-2)}(v^{-1}-v)u_{\mathscr{O}((k+r-1)\vec{c})}\xi^i_{1}
-v^{-(r-1)}u_{\mathscr{O}((k+r)\vec{c})}.\ \ \ \ \ \ \ \ \ \ \ \ \ \ \ \ (\ast 2)
\end{split}
\end{equation*}

Now we prove the claim by induction on $r$.

When $r=1$, we have
$\xi^i_1=u_{S^i_1}u_{S^i_0(p_{i}-1)}-v^2u_{S^i_0(p_i-1)}u_{S^i_1}$.

Also we have
$$u_{S^i_1}u_{\mathscr{O}(k\vec{c})}=v^{-1}(u_{\mathscr{O}(k\vec{c})\oplus
 S^i_1}+u_{\mathscr{O}(k\vec{c}+\vec{x_i})}),$$
$$u_{\mathscr{O}(k\vec{c})}u_{S^i_1}=u_{\mathscr{O}(k\vec{c})\oplus
 S^i_1},$$
$$u_{S^i_0(p_i-1)}u_{\mathscr{O}(k\vec{c}+\vec{x}_i)}=
 v^{-1}(u_{\mathscr{O}(k\vec{c}+\vec{x}_i)\oplus
 S^i_0(p_i-1)}+u_{\mathscr{O}((k+1)\vec{c})}),$$
$$u_{\mathscr{O}(k\vec{c}+\vec{x}_i)}u_{S^i_0(p_i-1)}
=u_{\mathscr{O}(k\vec{c}+\vec{x}_i)\oplus
 S^i_0(p_i-1)}.$$

Thus we deduce that
\begin{equation*}
\begin{split}
\xi^i_1u_{\mathscr{O}(k\vec{c})}&=
 vu_{S^i_1}u_{\mathscr{O}(k\vec{c})}u_{S^i_0(p_i-1)}
 -v^2u_{S^i_0(p_i-1)}u_{S^i_1}u_{\mathscr{O}(k\vec{c})}\\
 &=u_{\mathscr{O}(k\vec{c})}
 u_{S^i_1}u_{S^i_0(p_i-1)}+u_{\mathscr{O}(k\vec{c}+\vec{x}_i)}u_{S^i_0(p_i-1)}\\
 &\ \ \ \ \ \ \ \ -vu_{S^i_0(p_i-1)}u_{\mathscr{O}(k\vec{c})}
 u_{S^i_1}-vu_{S^i_0(p_i-1)}u_{\mathscr{O}(k\vec{c}+\vec{x}_i)}\\
 &=u_{\mathscr{O}(k\vec{c})}u_{S^i_1}u_{S^i_0(p_i-1)}
 +u_{\mathscr{O}(k\vec{c}+\vec{x}_i)}u_{S^i_0(p_i-1)}
 -v^2u_{\mathscr{O}(k\vec{c})}u_{S^i_0(p_i-1)}u_{S^i_1}\\
 &\ \ \ \ \ \ \ \ -u_{\mathscr{O}(k\vec{c}+\vec{x}_i)}u_{S^i_0(p_i-1)}-
 u_{\mathscr{O}((k+1)\vec{c})}\\
 &=u_{\mathscr{O}(k\vec{c})}\xi^i_1-u_{\mathscr{O}((k+1)\vec{c})}.
\end{split}
\end{equation*}

Then we assume that ($\ast 2$) holds for $r-1$. By \cite{S} 4.13, we have
\begin{equation*}
\begin{split}
\xi^i_ru_{\mathscr{O}(k\vec{c})}&=
(u_{\mathscr{O}(k\vec{c})}\xi^i_{r-1}+\xi^i_{r-2}u_{\mathscr{O}((k+1)\vec{c})}
-vu_{\mathscr{O}((k+1)\vec{c})}\xi^i_{r-2})u_{S^i_0}\\
&\ \ \ \ \ \ \ \ \
-vu_{S^i_0}(u_{\mathscr{O}(k\vec{c})}\xi^i_{r-1}+\xi^i_{r-2}u_{\mathscr{O}((k+1)\vec{c})}
-vu_{\mathscr{O}((k+1)\vec{c})}\xi^i_{r-2})\\
&=u_{\mathscr{O}(k\vec{c})}\xi^i_r
-vu_{\mathscr{O}((k+1)\vec{c})}\xi^i_{r-1}+v^{-1}\xi^i_{r-1}u_{\mathscr{O}((k+1)\vec{c})}.
\end{split}
\end{equation*}

This completes the proof of ($\ast 2$). And the lemma is a consequence of ($\ast 2$) and ($\ast 1$).
\end{proof}

\begin{Lem}\label{lem x^+_ij x^+ast and h_ij x^+ast j geq 2}
(1). $[x^+_{[i,j],l},x^+_{*,k}]=0$, for $1\leq i\leq n, 2\leq j\leq p_i-1$, $l\geq 0$ and
$k\in\mathbb{Z}$.

(2). $[h_{[i,j],l},x^+_{*,k}]=0$, for $1\leq i\leq n, 2\leq j\leq p_i-1$, $l\in\mathbb{N}$
and $k\in\mathbb{Z}$.
\end{Lem}

\begin{proof}
First we ovserve that (2) is a consequence of (1), since
$[h_{[i,j],l},x^+_{*,k}]=0$ if and only if
$[[x^{+}_{[i,j],l-1},x^{-}_{[i,j],1}],x^{+}_{*,k}]=0$ and we know that
$[x^-_{[i,j],1}],x^+_{*,k}]=0$ by Lemma \ref{lem hsi-}.

Note that $[x^+_{[i,j],0},x^+_{*,k}]=0$, because there is no non-trivial
extension between $\mathscr{O}(k\vec{c})$ and $S^i_j$ for $j\geq 2$.

Now we argue by induction on $j$. When $j=2$, using Lemma \ref{lem h_i1 x^+_ast} we have
\begin{equation*}
\begin{split}
[x^+_{[i,2],l},x^+_{*,k}]&=\frac{-l}{[l]}[[h_{[i,1],l},u_{S^i_2}],x^+_{*,k}]\\
&=\frac{-l}{[l]}[[h_{[i,1],l},x^+_{*,k}],u_{S^i_2}]=0.
\end{split}
\end{equation*}

Assume that for $1<j<m$ the relation holds, we also have
$[h_{[i,j],l},x^+_{*,k}]=0$ for $1<j<m$. Hence
\begin{equation*}
\begin{split}
[x^+_{[i,m],l},x^+_{*,k}]&=\frac{-l}{[l]}[[h_{[i,m-1],l},x^+_{[i,m],0}],x^+_{*,k}]\\
&=\frac{-l}{[l]}[[h_{[i,m-1],l},x^+_{*,k}],x^+_{[i,m],0}]=0.
\end{split}
\end{equation*}
\end{proof}

\begin{Lem}
$[x^-_{[i,j],l},x^+_{*,k}]=0$, for any $1\leq i\leq n,1\leq j\leq p_i-1,l\geq 1$ and
$k\in\mathbb{Z}$.
\end{Lem}

\begin{proof}
For the case $l=1$ this is just Lemma \ref{lem hsi-}. For $l>1$ we apply $\ad(h_{[i,j],r})$ and use lemma \ref{lem h_i1
x^+_ast} and lemma \ref{lem x^+_ij x^+ast and h_ij x^+ast j geq 2} (2).
\end{proof}

\section{Relations in $\mathbf{DH}(\Coh(\mathbb{X}))$}\label{sec relation of the whole part}
We prove all the Drinfeld relations in this section. In the
last section we have proved the relations in
$\mathbf{H}(\Coh(\mathbb{X}))$. In the same way we can prove the relations in
$\mathbf{H}^{-}(\Coh(\mathbb{X}))$. So we focus
on the relations involving both positive and negative elements.

\subsection{}
We first investigate the comultiplication of $h_{\ast,r}$ and $u_{\mathscr{O}(k\vec{c})}$ ($r\geq 0, k\in\mathbb{Z}$) in details, which is crucial for our calculations.
By definition we have
$$\Delta(h_{*,r})=h_{*,r}\otimes 1+K_{r\delta}\otimes h_{*,r}+\sum_{0<[A],[B]<r\delta} f(A,B) u_AK_{[B]}\otimes u_B,$$
\begin{equation*}
\begin{split}
\Delta(u_{\mathscr{O}(k\vec{c})})&=
u_{\mathscr{O}(k\vec{c})}\otimes
1+\sum^\infty_{r=0}\theta_{*,r}K_{\alpha_*+(k-r)\delta}\otimes
u_{\mathscr{O}((k-r)\vec{c})}\\
&+\sum_{\vec{x}\in L(\mathbf{p})_{+},\vec{x}\neq t\delta,\forall t\in \mathbb{N}}\theta_{\vec{x}}K_{[\mathscr{O}((k-r)\vec{c}-\vec{x})]}\otimes
u_{\mathscr{O}((k-r)\vec{c}-\vec{x})},
\end{split}
\end{equation*}
where $\{\theta_{*,r}\}_{r\geq 1}$ is defined by the following generating series
$$\sum_{k\geq 0}\theta_{*,r}u^k=\exp((v-v^{-1})\sum^\infty_{k=1}h_{*,k}u^k).$$

In the following, for simplicity we will call $\sum_{0<[A],[B]<r\delta} f(A,B) u_AK_{[B]}\otimes u_B$ the remaining terms in $\Delta(h_{*,r})$ and
$\sum_{\vec{x}\in L(\mathbf{p})_{+},\vec{x}\neq t\delta,\forall t\in \mathbb{N}}\theta_{\vec{x}}K_{[\mathscr{O}((k-r)\vec{c}-\vec{x})]}\otimes
u_{\mathscr{O}((k-r)\vec{c}-\vec{x})}$ the remaining terms in $\Delta(u_{\mathscr{O}(k\vec{c})})$.

\subsection{}
In this subsection we prove the relation $[h_{s,r},h_{t,m}]=0$ for $1\leq s,t \leq n$ and $rm<0$. We assume that $m>0$ and $r<0$.

\begin{Lem}
For fixed $i$, we have
$$[\pi^{+i}_{l_1,k_1},\pi^{-i}_{l_2,k_2}]=0,$$ where $1\leq l_1,l_2\leq p_i$ and $k_1,k_2\in \mathbb{N}$.
\end{Lem}

\begin{proof}
For simplicity, we will omit $i$ and write $p$ for $p_i$. Also we
write $h_{j,k}$ for $h_{[i,j],k}$, for any $1\leq j\leq p_i-1$.

For $1\leq k\leq p-1$, we have the
relation$$[\pi^{+}_{p,k_1},h_{p-k,-k_2}]=0.$$

This implies
$$[\pi^{+}_{p,k_1},\pi^-_{p-k+1,k_2}-(v^{k_2}+v^{-k_2})\pi^-_{p-k,k_2}+\pi^-_{p-k-1,k_2}]=0.$$

Note that $[\pi^{+}_{p,k_1},\pi^-_{p,k_2}]=0$. Hence we get a system
of homogeneous linear equations $AX=0$, where
 \[A=
 \left( \begin{array}{cccccc}a&1&0&0&...&0\\1&a&1&0&...&0\\0&1&a&1&...&0\\
 ...&...&...&...&...&...\\
 0&0&...&1&a&1\\
 0&0&...&0&1&a
 \end{array} \right)
 \]

$$X=(b_{p,p-1},b_{p,p-2},\cdots,b_{p,2},b_{p,1})^{t},$$
and $a=-(v^{k_2}+v^{-k_2})$, $b_{r,l}=[\pi^{+}_{p,k_1},\pi^-_{l,k_2}]$.

It is clear that $A$ is a nonsingular matrix, thus
$[\pi^{+}_{p,k_1},\pi^-_{l,k_2}]=0$ holds for any $1\leq l\leq p-1$.

Similarly, the relations
$[h_{l_1,k_1},h_{l_2,-k_2}]=0$ ($1\leq l_1\leq p-1$, $1\leq l_2\leq
p-1$) induce a system of homogeneous linear equations with
$(p-1)\times (p-1)$ variables $[\pi^{+}_{l_1,k_1},\pi^-_{l_2,k_2}]$
and the coefficient matrix is also nonsingular. Therefore, for any
$1\leq l_1\leq p-1,1\leq l_2\leq p-1$ and $k_1,k_2\in\mathbb{N}$, we have
$[\pi^{+}_{l_1,k_1},\pi^-_{l_2,k_2}]=0$.
\end{proof}

Now we can deduce that
$$[h_{*,r},h_{*,m}]=-\sum^{n}_{i=1}[\pi^{-i}_{1,-r},\pi^{+i}_{1,m}]=0,$$
$$[h_{*,r},h_{[i,j],m}]=-[\pi^{-i}_{1,-r},\pi^{+i}_{j+1,m}-(v^m+v^{-m})\pi^{-i}_{j,m}+\pi^{-i}_{j-1,m}]=0.$$

\subsection{}
In this subsection we deal with the following relation for the remaining
cases
$$[h_{s,r},x^\pm_{t,m}]=\pm \frac{1}{r}[la_{st}]x^\pm_{t,r+m}.$$

\begin{Lem}\label{lem comult considered 1}
$[h_{\ast,-l},x^{+}_{\ast,k}]=\frac{1}{l}[2l]x^{+}_{\ast,-l+k}.$ For
any $l>0$ and $k\in\mathbb{Z}$.
\end{Lem}

\begin{proof}
Recall that
$$\Delta(u_{\mathscr{O}(k\vec{c})})=
u_{\mathscr{O}(k\vec{c})}\otimes
1+\sum^\infty_{r=0}\theta_{*,r}K_{\alpha_*+(k-r)\delta}\otimes
u_{\mathscr{O}((k-r)\vec{c})}+\text{remaining terms}.$$

We prove that if $A$ is a
quotient of $\mathscr{O}(k\vec{c})$ as well as a subsheaf of a torsion sheaf occurring with nonzero coefficient in $h_{\ast,l}$, then
$[A]=s\delta$ for some $s\in\mathbb{N}$.

Otherwise, assume that $[A]$ is not a multiple of $\delta$, then there exists an exact sequence
$$0\rightarrow \mathscr{O}(\vec{x})\rightarrow
\mathscr{O}(k\vec{c})\rightarrow A\rightarrow 0,$$
where $\vec{x}$ is not a multiple of $\vec{c}$.

This implies
$$\Ext^1(S^m_{1},\mathscr{O}(\vec{x}))=\Hom(\mathscr{O}(\vec{x}),S^m_{0})=0,\ \text{for some }1\leq m\leq n.$$

There is always an injective
map from $S^m_{1}$ to $A$. Thus we have the following commutative
diagram where the right square is a pull-back diagram:
\begin{displaymath}
\xymatrix{
0\ar[r]&\mathscr{O}(\vec{x})\ar[r]\ar[d]&\mathscr{O}(\vec{x})\oplus S^m_{1}\ar[d]^f\ar[r]^-{p_2}&S^m_{1}\ar[d]^{i}\ar[r]&0\\
0\ar[r]&\mathscr{O}(\vec{x})\ar[r]&\mathscr{O}(k\vec{c})\ar[r]^-{g}&A\ar[r]&0
}
\end{displaymath}

Since there is no non-zero morphism from $S^m_{1}$ to
$\mathscr{O}(k\vec{c})$, we have $gf=0$. But $ip_2$ is clear not
zero, which is a contradiction.

This means we do not need to consider the remaining terms in $\Delta(u_{\mathscr{O}(k\vec{c})})$.

Now using the definition of the Drinfeld double and note that
$(\theta_{*,r},h_{*,r})=\frac{[2r]}{r}$, the lemma follows.
\end{proof}

\begin{Lem}
(1). $[h_{*,l},x^+_{[i,1],k}]= \frac{1}{l}[-l]x^+_{[i,1],k+l}$, for $l<0$, $1\leq i\leq n$ and $k\geq 0$.

(2). $[h_{*,l},x^-_{[i,1],k}]= -\frac{1}{l}[-l]x^-_{[i,1],k+l}$, for $l<0$ $1\leq i\leq n$ and $k\geq 1$.
\end{Lem}

\begin{proof}
(1). We have shown that $[\pi^{+i}_{2,-l},x^-_{[i,1],1}]=0$. Similarly we have
$$[\pi^{-i}_{2,-l},x^+_{[i,1],-1}]=0.$$

Apply $\ad(h_{[i,1],k+1})$ to the above formula, we have $[\pi^{-i}_{2,-l},x^+_{[i,1],k}]=0$.

So we can deduce that
\begin{equation*}
[h_{*,l},x^+_{[i,1],k}]=[\pi^{-i}_{1,-l},x^+_{[i,1],k}]=\frac{-1}{v^l+v^{-l}}[h_{[i,1],l},x^+_{[i,1],k}]=\frac{1}{l}[-l]x^+_{[i,1],k+l}.
\end{equation*}

(2). We have shown that $[\pi^{+i}_{2,-l},x^+_{[i,1],0}]=0$. Similarly we have
$$[\pi^{-i}_{2,-l},x^-_{[i,1],0}]=0.$$

Apply $\ad(h_{[i,1],k})$ to the above formula, we have $[\pi^{-i}_{2,-l},x^-_{[i,1],k}]=0$.

Hence we have
\begin{equation*}
[h_{*,l},x^-_{[i,1],k}]=[\pi^{-i}_{1,-l},x^-_{[i,1],k}]=\frac{-1}{v^l+v^{-l}}[h_{[i,1],l},x^-_{[i,1],k}]=-\frac{1}{l}[-l]x^-_{[i,1],k+l}.
\end{equation*}
\end{proof}

\begin{Lem}
(1). $[h_{*,l},x^+_{[i,m],k}]=0$, for $1\leq i\leq n, 2\leq m\leq p_i-1$, $l<0$ and $k\geq 0$.

(2). $[h_{*,l},x^-_{[i,m],k}]=0$, for $1\leq i\leq n,2\leq m\leq p_i-1$, $l<0$ and $k\geq 1$.
\end{Lem}

\begin{proof}
For the first equation, apply $\ad(h_{[i,m],k+1})$ to
$[h_{*,l},x^+_{[i,m],-1}]=0$. And for the second one, just apply
$\ad(h_{[i,m],k})$ to $[h_{*,l},x^-_{[i,m],0}]=0$.
\end{proof}

\begin{Lem}
$[h_{[i,1],l},x^+_{*,k}]=\frac{1}{l}[-l]x^+_{*,k+l}$, for $1\leq i\leq n,l<0$ and
$k\in\mathbb{Z}$.
\end{Lem}

\begin{proof}
Set $\xi^i_l=[x^+_{[i,1],-l},x^-_{[i,1],2l}]K_{[i,1]}$. We
have shown that
$$[x^+_{[i,1],-l},x^-_{[i,1],2l}]=\frac{\varphi_{[i,1],l}}{v-v^{-1}}.$$

By the definition of $\varphi$ (see \ref{subsec quantum loop alg} (7)), we have for any $r<0$,
$$-rh_{[i,1],r}=-r\xi^i_l+\sum^{-r-1}_{s=1}(v-v^{-1})sh_{[i,1],-s}\xi^i_{r+s}.\eqno(\ast 3)$$

We claim that the following identity holds:
\begin{equation*}
\begin{split}
\xi^i_lx^+_{*,k}&=x^+_{*,k}\xi^i_l+(v-v^{-1})x^+_{*,k-1}\xi^i_{l+1}
+v(v-v^{-1})x^+_{*,k-2}\xi^i_{l+2}+\\&\cdots
+v^{(-r-2)}(v-v^{-1})x^+_{*,k+l+1}\xi^i_{-1} -v^{(-r-1)}x^+_{*,k+l}.\ \ \ \
\ \ \ \ \ \ \ \ \ \ \ \ \ \ \ \ \ \ \ \ \ \ \ \ (\ast 4)
\end{split}
\end{equation*}

We prove the claim by induction. When $l=-1$, we have
\begin{equation*}
\begin{split}
\xi^i_{-1}x^+_{*,k}=&[x^+_{[i,1],0},x^-_{[i,1],-1}]K_{[i,1]}x^+_{*,k}\\
&=v^{-1}[x^+_{[i,1],0},x^-_{[i,1],-1}]x^+_{*,k}K_{[i,1]}\\
&=x^+_{*,k}\xi^i_{-1}
-v^{-1}x^+_{*,k-1}\xi^i_{0}+v\xi^i_{0}x^+_{*,k-1}\\
&=x^+_{*,k}\xi^i_{-1}+x^+_{*,k-1}.
\end{split}
\end{equation*}

Now assume ($\ast 4$) holds for $l=-r+1$. We will prove the case $l=-r$. By \cite{S} 4.13 and $[x^-_{[i,1],2l},x^+_{*,k}]=0$,
which will be proved in Lemma \ref{lem proved later} independent of this
lemma, we deduce that
\begin{equation*}
\begin{split}
\xi^i_lx^+_{*,k}&=[x^+_{[i,1],-l},x^-_{[i,1],2l}]K_{[i,1]}x^+_{*,k}\\
&=v^{-1}[x^+_{[i,1],-l},x^-_{[i,1],2l}]x^+_{*,k}K_{[i,1]}\\
&=v^{-1}(x^+_{[i,1],-l}x^-_{[i,1],2l}x^+_{*,k}-x^-_{[i,1],2l}x^+_{[i,1],-l}x^+_{*,k})K_{[i,1]}\\
&=v^{-1}(x^+_{[i,1],-l}x^+_{*,k}x^-_{[i,1],2l}-x^-_{[i,1],2l}x^+_{[i,1],-l}x^+_{*,k})K_{[i,1]}\\
&=v^{-1}((vx^+_{[i,1],-l+1}x^+_{*,k-1}+vx^+_{*,k}x^+_{[i,1],-l}-x^+_{*,k-1}x^+_{[i,1],-l+1})x^-_{[i,1],2l}\\
&\ \ \ \ \ \ \-x^-_{[i,1],2l}(vx^+_{[i,1],-l+1}x^+_{*,k-1}+vx^+_{*,k}x^+_{[i,1],-l}-x^+_{*,k-1}x^+_{[i,1],-l+1}))K_{[i,1]}\\
&=x^+_{*,k}\xi^i_l-v^{-1}x^+_{*,k-1}\xi^i_{l+1}+v\xi^i_{l+1}x^+_{*,k-1}.
\end{split}
\end{equation*}

The lemma is now a consequence of ($\ast 3$) and ($\ast 4$).
\end{proof}

\begin{Lem}
$[h_{[i,m],l},x^+_{*,k}]=0$, for $1\leq i\leq n,2\leq m\leq p_i-1$, $l<0$ and
$k\in\mathbb{Z}$.
\end{Lem}

\begin{proof}
When $m=2$, we have proved that
$$[x^-_{[i,2],1},x^+_{*,k}]=0,\ [x^+_{[i,2],0},x^+_{*,k}]=0.$$

Apply $\ad(h_{[i,1],l-1})$ to the second formula, we get
$$[x^+_{[i,2],l-1},x^+_{*,k}]=0.$$

Then we have $[\varphi_{[i,2],l},x^+_{*,k}]=0$, which is the same as $[h_{[i,2],l},x^+_{*,k}]=0$.

Assume for any $m<p_{i}$ the relation $[x^+_{[i,m],l},x^+_{*,k}]=0$
holds. We shall prove for the case $m+1$.

Note that the following still holds:
$$[x^-_{[i,m+1],1},x^+_{*,k}]=0,\ [x^+_{[i,m+1],0},x^+_{*,k}]=0$$

Now apply $\ad(h_{[i,m],l-1})$ to the second formula, we get
$$[x^+_{[i,m+1],l-1},x^+_{*,k}]=0.$$

Then we deduce that $[\varphi_{[i,m+1],l},x^+_{*,k}]=0$, which is the same as $[h_{[i,m+1],l},x^+_{*,k}]=0$.
\end{proof}

\subsection{}
Next we consider the relation
$$[x^+_{s,k},x^-_{t,l}]=\delta_{st}\frac{\psi_{s,k+l}-\varphi_{s,k+l}}{v-v^{-1}}.$$

\begin{Lem}\label{lem comult considered 2}
For any $k,l$ such that $k+l\geq 0$,
$$[x^+_{*,k},x^-_{*,l}]=\frac{\psi_{*,k+l}-\varphi_{*,k+l}}{v-v^{-1}}.$$
\end{Lem}

\begin{proof}
Recall the comultiplication:
$$\Delta(u_{\mathscr{O}(k\vec{c})})=
u_{\mathscr{O}(k\vec{c})}\otimes
1+\sum^\infty_{r=0}\theta_{*,r}K_{\alpha_*+(k-r)\delta}\otimes
u_{\mathscr{O}((k-r)\vec{c})}+\text{remaining terms}.$$

$$\Delta(u_{\mathscr{O}(-l\vec{c})})=
u_{\mathscr{O}(-l\vec{c})}\otimes
1+\sum^\infty_{r=0}\theta_{*,r}K_{\alpha_*+(-l-r)\delta}\otimes
u_{\mathscr{O}((-l-r)\vec{c})}+\text{remaining terms}.$$

For any $A_1=u_{B_1}K_{[C_1]}\otimes u_{C_1}$ (resp. $A_2=u_{B_2}K_{[C_2]}\otimes u_{C_2}$) appearing in the remaining
terms of $\Delta(u_{\mathscr{O}(k\vec{c})})$ (resp. $\Delta(u_{\mathscr{O}(-l\vec{c})})$), $B_1$ is a nonzero sheaf of
finite length and $C_2$ is a nonzero line bundle. So they are not
isomorphic to each other. And similarly, $C_1$ is a nonzero line
bundle and $B_{2}$ is a nonzero sheaf of finite length. They are not
isomorphic to each other. Thus we do not need to consider the remaining terms.

Then the lemma can be deduced by the definition of the Drinfeld double.
\end{proof}

\begin{Lem}\label{lem proved later}
$[x^+_{*,k},x^-_{[i,j],l}]=0$, for any $k\in\mathbb{Z}$, $1\leq i\leq n,1\leq j\leq p_i-1$,$l\leq 0$.
\end{Lem}

\begin{proof}
By Lemma \ref{lem hsi-}, we have $[x^+_{*,k},x^-_{[i,j],1}]=0$.

For $j=1$, we have $$[x^+_{*,k},x^-_{[i,1],l}]=0,$$
by applying $\ad(h_{*,l-1})$ to $[x^+_{*,k},x^-_{[i,1],1}]=0$.

For $2\leq j\leq p_{i}-1$, we deduce $[x^+_{*,k},x^-_{[i,j],l}]=0$ by applying $\ad(h_{[i,j-1],l-1})$ to $[x^+_{*,k},x^-_{[i,j],1}]=0$.
\end{proof}

\subsection{}
Now we consider the following relation:
$$x^\pm_{s,k+1}x^\pm_{t,l}-v^{\pm a_{st}}x^\pm_{t,l}x^\pm_{s,k+1}=v^{\pm
a_{st}}x^\pm_{s,k}x^\pm_{t,l+1}-x^\pm_{t,l+1}x^\pm_{s,k}.$$

\begin{Lem}
For any $1\leq i\leq n,l<0$ and $k\in \mathbb{Z}$,
$$x^+_{*,k+1}x^+_{[i,1],l}-v^{-1}x^+_{[i,1],l}x^+_{*,k+1}
=v^{-1}x^+_{*,k}x^+_{[i,1],l+1}-x^+_{[i,1],l+1}x^+_{*,k}.$$
\end{Lem}

\begin{proof}
We have already known that
$$x^+_{*,k+1}x^+_{[i,1],0}-v^{-1}x^+_{[i,1],0}x^+_{*,k+1}
=v^{-1}x^+_{*,k}x^+_{[i,1],1}-x^+_{[i,1],1}x^+_{*,k}.$$

The required result can be obtained by applying $\ad(h_{[i,1],l-1})$ to the above formula.
\end{proof}

\begin{Lem}
$[x^+_{*,k},x^+_{[i,j],l}]=0$ for $1\leq i\leq n, 2\leq j\leq p_i-1$, $l<0$ and
$k\in\mathbb{Z}$.
\end{Lem}

\begin{proof}
Just apply $\ad(h_{[i,1],l-1})$ to $[x^+_{*,k},x^+_{[i,j],0}]=0$.
\end{proof}

\subsection{}
Finally we deal with the following relation:
$$\Sym_{k_1,....,k_n}\sum^n_{t=0}(-1)^t\left[ \begin{array}{c} n \\
t\\ \end{array}\right]x^\pm_{i,k_1}\cdots
x^\pm_{i,k_t}x^\pm_{j,l}x^\pm_{i,k_{t+1}}\cdots x^\pm_{i,k_n}=0,$$
where $i\neq j$ and $n=1-a_{ij}$.

\begin{Lem}
For any $k_1\leq 0$,$1\leq i\leq n,k_2,t\in\mathbb{Z}$
$$\Sym_{k_1,k_2}\{ x^+_{[i,1],k_1}x^+_{[i,1],k_2}x^+_{*,t}
-[2]x^+_{[i,1],k_1}x^+_{*,t}x^+_{[i,1],k_2}
+x^+_{*,t}x^+_{[i,1],k_1}x^+_{[i,1],k_2}\}=0.$$
\end{Lem}

\begin{proof}
For $k_1=0$, $k_2=0$ and $t\in \mb{Z}$, we know that
$$\Sym_{0,0}\{ x^+_{[i,1],0}x^+_{[i,1],0}x^+_{*,t}
-[2]x^+_{[i,1],0}x^+_{*,t}x^+_{[i,1],0}
+x^+_{*,t}x^+_{[i,1],0}x^+_{[i,1],0}\}=0.$$

Apply $\ad(h_{*,k_1})$, we get
\begin{equation*}
\begin{split}
0=&-\frac{[k_1]}{k_1}\Sym_{k_1,0}\{
x^+_{[i,1],k_1}x^+_{[i,1],0}x^+_{*,t}
-[2]x^+_{[i,1],k_1}x^+_{*,t}x^+_{[i,1],0}\\
&\ \ \ \ \ \ \ \ \ \ +x^+_{*,t}x^+_{[i,1],k_1}x^+_{[i,1],0}\}\\
&+\frac{[2k_1]}{k_1}\Sym_{0,0}\{
x^+_{[i,1],0}x^+_{[i,1],0}x^+_{*,t+k_1}
-[2]x^+_{[i,1],0}x^+_{*,t+k_1}x^+_{[i,1],0}\\
&\ \ \ \ \ \ \ \ \ \ +x^+_{*,t+k_1}x^+_{[i,1],0}x^+_{[i,1],0}\}.
\end{split}
\end{equation*}

Hence we have
$$\Sym_{k_1,0}\{x^+_{[i,1],k_1}x^+_{[i,1],0}x^+_{*,t}
-[2]x^+_{[i,1],k_1}x^+_{*,t}x^+_{[i,1],0}
+x^+_{*,t}x^+_{[i,1],k_1}x^+_{[i,1],0}\}=0.$$

Then apply $\ad(h_{*,k_2})$ to the above equation, which yields
\begin{equation*}
\begin{split}
0=&-\frac{[k_2]}{k_2}\Sym_{k_1+k_2,0}\{
x^+_{[i,1],k_1+k_2}x^+_{[i,1],0}x^+_{*,t}-[2]x^+_{[i,1],k_1+k_2}x^+_{*,t}x^+_{[i,1],0}\\
&\ \ \ \ \ \ \ \ \ \ +x^+_{*,t}x^+_{[i,1],k_1+k_2}x^+_{[i,1],0}\}\\
&+\frac{[2k_2]}{k_2}\Sym_{k_1,0}\{
x^+_{[i,1],k_1}x^+_{[i,1],0}x^+_{*,t+k_2}
-[2]x^+_{[i,1],k_1}x^+_{*,t+k_2}x^+_{[i,1],0}\\
&\ \ \ \ \ \ \ \ \ \ +x^+_{*,t+k_2}x^+_{[i,1],k_1}x^+_{[i,1],0}\}\\
&-\frac{[k_2]}{k_2}\Sym_{k_1,k_2}\{
x^+_{[i,1],k_1}x^+_{[i,1],k_2}x^+_{*,t}
-[2]x^+_{[i,1],k_1}x^+_{*,t}x^+_{[i,1],k_2}\\
&\ \ \ \ \ \ \ \ \ \ +x^+_{*,t}x^+_{[i,1],k_1}x^+_{[i,1],k_2}\}.
\end{split}
\end{equation*}

The proof is completed.
\end{proof}

\begin{Lem}
For any $r<0$,$1\leq i\leq n,t_1,t_2\in\mathbb{Z}$
$$\Sym_{t_1,t_2}\{ x^+_{*,t_1}x^+_{*,t_2}x^+_{[i,1],r}
-[2]x^+_{*,t_1}x^+_{[i,1],r}x^+_{*,t_2}
+x^+_{[i,1],r}x^+_{*,t_1}x^+_{*,t_2}\}=0.$$
\end{Lem}

\begin{proof}
For any $t_1,$ $t_2$, the following holds:
$$\Sym_{t_1,t_2}\{
x^+_{*,t_1}x^+_{*,t_2}x^+_{[i,1],0}
-[2]x^+_{*,t_1}x^+_{[i,1],0}x^+_{*,t_2}
+x^+_{[i,1],0}x^+_{*,t_1}x^+_{*,t_2}\}=0.$$

Now apply $\ad(h_{*,l})$, we get
\begin{equation*}
\begin{split}
0=&-\frac{[l]}{l}\Sym_{t_1+l,t_2}\{
x^+_{*,t_1+l}x^+_{*,t_2}x^+_{[i,1],0}
-[2]x^+_{*,t_1+l}x^+_{[i,1],0}x^+_{*,t_2}
+x^+_{[i,1],0}x^+_{*,t_1+l}x^+_{*,t_2}\}\\
&-\frac{[l]}{l}\Sym_{t_1,t_2+l}\{
x^+_{*,t_1}x^+_{*,t_2+l}x^+_{[i,1],0}
-[2]x^+_{*,t_1}x^+_{[i,1],0}x^+_{*,t_2+l}
+x^+_{[i,1],0}x^+_{*,t_1}x^+_{*,t_2+l}\}\\
&+\frac{[2l]}{l}\Sym_{t_1,t_2}\{ x^+_{*,t_1}x^+_{*,t_2}x^+_{[i,1],l}
-[2]x^+_{*,t_1}x^+_{[i,1],l}x^+_{*,t_2}
+x^+_{[i,1],l}x^+_{*,t_1}x^+_{*,t_2}\}.
\end{split}
\end{equation*}

Thus the relation holds for all $l<0$.
\end{proof}

\section{Remarks on derived equivalence and PBW-basis}\label{sec remarks}
In this section we restrict to the case that $\mathfrak{g}$ is of
finite type, i.e. $\mathcal{L}\mathfrak{g}$ is an affine Kac-Moody
algebra.

\subsection{Derived equivalences and double Hall algebras}\label{subsec derived eq}
In this case, the associated star-shaped Dynkin diagram $\Gamma$ is
of type A-D-E. Denote by $\widehat{\Gamma}$ the corresponding extended Dynkin
diagram. We know that the category
$\Coh(\mathbb{X})$ is derived equivalent to $\modl\Lambda$ where
$\Lambda$ is the path algebra of $\widehat{\Gamma}$ (hence $\Lambda$
is a tame hereditary algebra). More precisely, let $\mu$ be the
\emph{slope} function for coherent sheaves and $\chi$ be the
\emph{Euler characteristic} of the weighted projective line
$\mathbb{X}$ (see \cite{GL} for missing definitions). We
have the following theorem:

\begin{Thm}[\cite{GL}]\label{thm derived eq}
The direct sum $T$ of a representative system of indecomposable
bundles $E$ with slope $0\leq \mu(E)<\chi$ is a tilting object of
$\Coh(\mathbb{X})$ whose endomorphism ring is isomorphic to
$\Lambda$. Thus we
have$$\mathscr{D}^{b}(\Coh(\mathbb{X}))\simeq\mathscr{D}^{b}(\modl\Lambda).$$
\end{Thm}

Recently Cramer has proved the following result, which asserts that
the double Hall algebra is invariant under derived equivalences.

\begin{Thm}[\cite{C}]\label{thm Cramer iso}
Let $\mathscr{A}$ and $\mathscr{B}$ be two $k-$linear finitary
hereditary categories. Assume that there is
an equivalence of triangulated categories
$D^b(\mathscr{A})\xrightarrow{\mathbb{F}}D^b(\mathscr{B})$. Let
$R(\mathscr{A})\xrightarrow{\widehat{\mathbb{F}}}R(\mathscr{B})$be
the induced equivalence of the root categories. Then there is an
algebra isomorphism. $\mathbb{F}:\dh(\mathscr{A})\longrightarrow
\dh(\mathscr{B})$ uniquely determined by the following property. For
any object $X$ in $\mathscr{A}$ such that $\mathbb{F}(X)\simeq
\widehat{X}[-n_{\mathbb{F}}(X)]$ with $\widehat{X}$ in
$\mathscr{B}$ and $n_{\mathbb{F}}(X)\in \mathbb{Z}$ we have:\\
$$\mathbb{F}(u^{\pm}_{X})=v^{-n_{\mathbb{F}}(X)<\widehat{X},\widehat{X}>}u_{\widehat{X}}^{\pm \overline{n_{\mathbb{F}}(X)}}K^{\pm
n_{\mathbb{F}}(X)}_{\widehat{X}}$$ \\
where $\overline{n_{\mathbb{F}}(X)}=+$ (resp. $-$) if $n_{\mathbb{F}}(X)$ is
even (resp. odd). For $\alpha\in
K_0(\mathscr{A})$, we have
$\mathbb{F}(K_\alpha)=K_{\mathbb{F}(\alpha)}$.
\end{Thm}

\subsection{Incompatibility of some homomorphisms}\label{subsec compatibility of iso}
Now let us consider the following diagram
\begin{equation*}
\xymatrix{
  U_{v}(\hat{\mathfrak{g}}) \ar[d]_{\psi} \ar[r]^-{\Omega} & \mathbf{DC}(\modl\Lambda) \ar[r]^{\iota_{1}} & \mathbf{DH}(\modl\Lambda) \ar[d]^{\mathbb{F}} \\
  U_{v}(\mathcal{L}\mathfrak{g}) \ar[r]^-{\Xi}            & \mathbf{DC}(\Coh(\mathbb{X})) \ar[r]^{\iota_{2}}        & \mathbf{DH}(\Coh(\mathbb{X}))   }
\end{equation*}
where $\psi$ is the Drinfeld-Beck isomorphism, $\Omega$ is the
isomorphism in Theorem \ref{thm Xiao}, $\Xi$ is the epimorphism
given by Theorem\ref{thm main}, $\iota_{1}$, $\iota_{2}$ are
natural embeddings and $\mathbb{F}$ is the isomorphism in Theorem \ref{thm Cramer iso}.

When $\mathfrak{g}=\mathfrak{sl}_{2}$, we know that $\Lambda$ is the
path algebra of the Kronecker quiver and $\mathbb{X}$ is the (non-weighted)
projective line $\mathbb{P}^{1}$. In this case it has been proved in
\cite{BS} that $\Xi$ is an isomorphism and the above diagram is commutative. This is equivalent
to say that the restriction of the isomorphism $\mathbb{F}$ to
$\mathbf{DC}(\modl\Lambda)$ gives the isomorphism
$$\Xi\circ\psi\circ\Omega^{-1}:\mathbf{DC}(\modl\Lambda)\simeq\mathbf{DC}(\Coh(\mathbb{X})).$$

However, for the other cases, the diagram may not be commutative even if $\Xi$ is an isomorphism. The reason is as follows:

Denote by $E_{m}$ the Chevalley generators of the standard positive part
$U_{v}^{+}(\hat{\mathfrak{g}})$. Here $m\in\Gamma_{0}\cup\{e\}$, where $e$ denotes the extending vertex of
$\widehat{\Gamma}$. By definition of $\Omega$, the image of each
$E_{m}$ in $\mathbf{DH}(\modl\Lambda)$ is a simple $\Lambda$-module.

On the other hand, the Drinfeld-Beck isomorphism $\psi$ sends
$E_{m}$ to $x_{m,0}^{+}$ for all $m$ except the extending vertex.
Now if $m$ is not the central vertex $\ast$, by Theorem \ref{thm main}, the image of
$x_{m,0}^{+}$ in $\mathbf{DH}(\Coh(\mathbb{X}))$ under the
homomorphism $\Xi$ is a simple sheaf lying on the bottom of
some non-homogeneous tube.

Thus if the diagram is commutative, we should have a derived equivalence functor
$\mathbb{G}:\mathscr{D}^{b}(\modl\Lambda)\simeq\mathscr{D}^{b}(\Coh(\mathbb{X}))$
assigning all simple $\Lambda$-modules, except two (corresponding to the vertex $\ast$ and $e$), to sheaves on the
bottom of non-homogeneous tubes. However, this is
impossible for type $D$ and $E$.

Note that in \cite{BS2} it has been proved that the isomorphism $\mathbb{F}$ restricts to an isomorphism
$\mathbf{DC}(\modl\Lambda)\simeq\mathbf{DC}(\Coh(\mathbb{X}))$, which we still denote by $\mathbb{F}$.
Then the composition $\mathbb{F}\circ\Omega\circ\psi^{-1}$ gives an isomorphism $U_{v}(\mathcal{L}\mathfrak{g})\simeq\mathbf{DC}(\Coh(\mathbb{X}))$. But it is difficult to explicitly find out Drinfeld's generators and relations in $\mathbf{DC}(\Coh(\mathbb{X}))$ through this isomorphism.

\subsection{Two PBW-type bases}\label{subsec PBW-basis}
Now we have two realizations of the quantum affine algebra
$U_{v}(\hat{\mathfrak{g}})$ arising from Hall algebras of two different
hereditary categories which are derived equivalent. Note that the
derived equivalence
$\mathscr{D}^{b}(\Coh(\mathbb{X}))\simeq\mathscr{D}^{b}(\modl\Lambda)$
is ``visible" if one looks at the Auslander-Reiten-quivers.
Recall that the AR-quiver of $\Coh(\mathbb{X})$ consisting of two
components, the locally free part $\mathscr{F}$ and the torsion part
$\mathscr{T}$. While there are three components in the AR-quiver of
$\modl\Lambda$, namely the preprojective component $\mathscr{P}$,
the preinjective one $\mathscr{I}$ and the regular one
$\mathscr{R}$. Roughly speaking, the derived equivalence is given by
splitting the component $\mathscr{F}$ into two pieces corresponding to $\mathscr{P}$ and
$\mathscr{I}[-1]$ respectively, and identifying $\mathscr{T}$ with
$\mathscr{R}$.

The Hall algebra approach has many advantages. For example, the Hall
algebra has a natural basis indexed by the isomorphism classes of
objects in the category. Moreover, the structure of the category (e.g. the AR-quiver) give us more information.
In particular, one can construct a PBW-type
basis encoding the structure of the category. This has been done in
\cite{LXZ} for the Hall algebra of $\modl\Lambda$.

\begin{Prop}[\cite{LXZ}]\label{prop PBW-basis module cat}
The following set of elements
$$\{u_{P}E_{\pi_{1}\mathbf{c}}
E_{\pi_{2}\mathbf{c}}\cdots E_{\pi_{n}\mathbf{c}}T_{\omega}u_{I}\}$$
is a basis of the composition algebra $\mathbf{C}(\modl\Lambda)$,
i.e. a basis of the (standard) positive part
$U_{v}^{+}(\hat{\mathfrak{g}})$.
\end{Prop}

Let us briefly explain the notations in the above proposition: $P$ runs over all
preprojective modules and $I$ runs over all preinjective modules.
Hence $u_{P}$, $u_{I}$ are basis elements arising from preprojective
and preinjective components respectively. For each $i$, the $E_{\pi_{i}\mathbf{c}}$ is a certain element in
$\mathbf{H}(\mathscr{T}_{\lambda_{i}})$, where $\mathscr{T}_{\lambda_{i}}$ is a non-homogeneous tube. These elements were first
constructed in \cite{DDX}, we omit the explicit definition here.
$\omega=(\omega_{1}\geq\omega_{2}\geq\cdots\geq\omega_{t})$ runs over all
partitions of positive integers. And
$T_{\omega}=T_{\omega_{1}}T_{\omega_{2}}\cdots T_{\omega_{t}}$, where
$T_{r}$ is the element defined in \ref{subsec element T_r}.
Note that the regular component $\mathscr{R}$ is equivalent to the
torsion part $\mathscr{T}$.

Similarly we can construct
a PBW-type basis for another positive part of $U_{v}(\hat{\mathfrak{g}})$ using the Hall algebra of $\Coh(\mathbb{X})$:

\begin{Prop}\label{prop PBW-basis coh sheaf}
The following set of elements
$$\{u_{V}E_{\pi_{1}\mathbf{c}}
E_{\pi_{2}\mathbf{c}}\cdots
E_{\pi_{n}\mathbf{c}}T_{\omega}\}$$
is a basis of the composition algebra $\mathbf{C}(\Coh(\mathbb{X}))$, i.e. a basis of
$U_{v}(\hat{\mathfrak{n}})$.
\end{Prop}

In this proposition, $V$ runs over all vector bundles in
$\mathscr{F}$. Other notations are the same as Proposition \ref{prop
PBW-basis module cat}.

Comparing the above two propositions, we can see that the PBW-type bases of two
different halves of the quantum affine algebra are related by the derived
equivalence of the two categories $\modl\Lambda$ and $\Coh(\mathbb{X})$.

\bigskip
\par\noindent {\bf Acknowledgments.}
We would like to thank Professor Olivier Schiffmann for his reading and remarks.

\end{document}